\documentclass[11pt, leqno]{article}
\usepackage{amsmath,amsfonts,amssymb,wasysym}
\usepackage[latin1]{inputenc}
\usepackage{shapepar}
\usepackage{color}
\usepackage{cases}
\usepackage{graphicx}
\newcommand{\R}{{\mathbb R}}

\newcommand{\ren}{{\mathbb R}^N}
\newcommand{\N}{{\mathbb N}}

\newcommand{\be}[1]{\begin{equation}\label{#1}}
\newcommand{\ee}{\end{equation}}

\newcommand{\prf}{\par\smallskip\noindent{\sl Proof. \/}}
\newcommand{\finprf}{\unskip\null\hfill$\;\square$\vskip 0.3cm}

\newenvironment{proof}{\prf}{\finprf}
\newtheorem{theorem}{Theorem}[section]
\newtheorem{lemma}[theorem]{Lemma}
\newtheorem{corollary}[theorem]{Corollary}
\newtheorem{proposition}[theorem]{Proposition}
\newtheorem{remark}[theorem]{Remark}
\newtheorem{definition}[theorem]{Definition}

\def\qed{\,\unskip\kern 6pt \penalty 500
\raise -2pt\hbox{\vrule \vbox to8pt{\hrule width 6pt
\vfill\hrule}\vrule}\par}
\definecolor{darkblue}{rgb}{0.05, .05, .65}
\definecolor{darkgreen}{rgb}{0.1, .65, .1}
\definecolor{darkred}{rgb}{0.8,0,0}

\begin{document}
\title{Ground States for Diffusion Dominated Free Energies with Logarithmic Interaction}
\author{Jos\'e Antonio Carrillo\footnote{Department of Mathematics, Imperial College London, SW7 2AZ London, United Kingdom. Email:
carrillo@imperial.ac.uk}
\and
Daniele Castorina\footnote{Dipartimento di Matematica, Universit\`{a} di Roma ``Tor Vergata'', Via della Ricerca Scientifica, 00133
Roma, Italy. Email: castorin@mat.uniroma2.it}
\and
Bruno Volzone\footnote{Dipartimento di Ingegneria, Universit\`{a} degli Studi di Napoli ``Parthenope'',
80143 Italia. E-mail: bruno.volzone@uniparthenope.it}}
\maketitle

\begin{abstract}
Replacing linear diffusion by a degenerate diffusion of porous medium type is known to regularize the classical two-dimensional parabolic-elliptic Keller-Segel model \cite{CalCar06}. The implications of nonlinear diffusion are that solutions exist globally and are uniformly bounded in time. We analyse the stationary case showing the existence of a unique, up to translation, global minimizer of the associated free energy. Furthermore, we prove that this global minimizer is a radially decreasing compactly supported continuous density function which is smooth inside its support, and it is characterized as the unique compactly supported stationary state of the evolution model. This unique profile is the clear candidate to describe the long time asymptotics of the diffusion dominated classical Keller-Segel model for general initial data.
\end{abstract}

\section{Introduction}

Ground states of free energies play a crucial role in the long
time asymptotics of nonlinear aggregation diffusion models. These
nonlocal partial differential equations are ubiquitous in the
mathematical modelling of phenomena which involve a large number of particles.
For instance, nonlocal drift-diffusion equations show up naturally in
semiconductor modelling, bacterial chemotaxis, granular media, and many others, see  \cite{J,BDP,CMV1} and the
references therein. These equations are just based on two
competing mechanisms, namely the attraction, modelled by a nonlocal
force, and the repulsion, modelled by a nonlinear diffusion.

One of the archetypical models of this type is the so-called
classical parabolic-elliptic Keller-Segel model. This model was classically introduced as
the simplest description for chemotatic bacteria movement in which
linear diffusion tendency to spread fights the attraction due to
the logarithmic kernel interaction in two dimensions. Despite the
large amount of literature in this field, many advances have been
done in the last ten years thanks to the combination of different
ideas ranging from functional inequalities to entropy-entropy
dissipation techniques passing through optimal transport. We refer
to \cite{BDP,BCM,BCC,CD} and the references therein for some aspects
of this {\em fair competition} case in which there is a well-defined
critical mass. In fact, here a clear dichotomy arises: if the total
mass of the system is less than the critical mass, then the long time
asymptotics are described by a self-similar solution, while for a
mass larger than the critical one, there is finite time blow-up.
For the exact critical mass case, a detailed study has also been
performed in \cite{BCM,BCC}.

The existence of a well-defined critical mass can be generalized
to more dimensions if one allows for degenerate diffusions. In
fact, let us consider the evolution of the probability density
$\rho$ given by
\begin{equation}
\rho_t=\Delta
\rho^{m}+\nabla\cdot\left(\rho(\nabla\mathcal{W}\ast\rho)\right)\quad\text{
in }\R^d
\label{KellerSg}
\end{equation}
where $d\geq 2$, and with the homogeneous kernel $\mathcal{W}(x) =
|x|^{\alpha}/\alpha$ with $-d<\alpha<0$. By choosing $\alpha=2-d$,
$d\geq 3$ and $m=2-2/d$, it was shown in \cite{CarBlanLau} that there
exist a critical mass and an exact dichotomy as in the classical
Keller-Segel model. This is just based on the fact that these
equations share a common structural setting, namely they have a
well-defined free energy functional given by
\begin{equation}\label{free}
E(\rho)=\frac1{m-1} \int_{\mathbb{R}^d} \rho^m(x) \,dx
+\frac{1}{2\alpha}\int_{\mathbb{R}^d}\int_{\mathbb{R}^d}|x-y|^{\alpha}\rho(x)\rho(y)
\,dx\,dy\,,
\end{equation}
and that the two terms in \eqref{free} scale equally by dilations
in that particular case. Actually, this fact is also satisfied by
all the {\em fair competition} cases in which $m=(d-\alpha)/d$.
Therefore, there is a well-defined critical mass in all the fair
competition cases by generalizing the arguments in \cite{BCC}.
While the analysis of the fair competition cases can be considered
advanced, it is not so for both the {\em diffusion dominated} case,
$m>(d-\alpha)/d$, and the {\em aggregation dominated}
case, $m<(d-\alpha)/d$. Regarding the latter, recent results in
\cite{BianLiu} discriminate blow-up and global existence depending on the
initial conditions and on the exponent $m$ for the particular
case of $\alpha=2-d$, $d\geq 3$. Other results in this direction also appear in
a series of papers by Sugiyama \cite{S1,S2,S3}. However, in the
diffusion dominated case, $m>(d-\alpha)/d$, there is little information about
the long time asymptotics, seemingly due to the lack of confinement by
the interaction kernel, see \cite{BianLiu,S2}. It is actually
proved that solutions exist globally, and that they are bounded
uniformly in time without further information about their behavior
at infinity. Existence of steady states in
the case $\alpha=2-d$, $d\geq 3$, $m>2-2/d$, has been analyzed in \cite{BianLiu}.

In this manuscript, we build up in the understanding of the long
time asymptotics in the classical diffusion dominated case in two dimensions.
Calvez and the first author proved in \cite{CalCar06} that solutions
corresponding to the classical diffusion dominated two-dimensional Keller Segel
model:
\begin{equation}
\rho_t=\Delta
\rho^{m}+\tfrac1{2\pi}\nabla\cdot\left(\rho(\nabla\log|x|\ast\rho)\right)\quad\text{
in }\R^2\,,\qquad \mbox{ with $m>1$,}\label{KellerS}
\end{equation}
exist globally and are uniformly
bounded in time. However, they were not able to clarify the long time asymptotics.
Here we encounter once again the structural setting of a free energy functional given by
\begin{align}
\mathsf{G}[\rho]:=\frac{1}{m-1}\int_{\R^{2}}\rho^{m}\,dx+\frac{1}{4\pi}\int_{\R^{2}}\int_{\R^{2}}\log|x-y|\,\rho(x)\rho(y)\,dx\,dy\,.
\label{eq.1}
\end{align}
Since the free energy $\mathsf{G}[\rho]$ is decaying in time through the evolution of the flow associated to \eqref{KellerS}, one may expect convergence towards
the possible (global) minimizers of $\mathsf{G}[\rho]$ over mass densities. Due to the translational invariance of \eqref{KellerS}, we will consider the set of
admissible functions for the functional $\mathsf{G}[\rho]$ as the set of \emph{zero center-of-mass densities}
\begin{equation*}
\mathcal{Y}_{M}:=\left\{\rho\in L_{+}^{1}(\R^{2})\cap L^{m}(\R^{2}):\,\|\rho\|_{1}=M,\int_{\R^{2}}x\rho(x)\,dx=0\right\}
\end{equation*}
for a given \emph{mass} $M>0$.

This work is entirely devoted to show the existence of a unique global minimizer of the free energy $\mathsf{G}[\rho]$ in
$\mathcal{Y}_{M}$. Furthermore, we will show that this global minimizer is a radially decreasing compactly supported continuous density function smooth inside
its support, and that is characterized as the unique (up to translations) compactly supported stationary state of the diffusion dominated Keller-Segel model \eqref{KellerS} with
$m>1$. This unique profile is the clear candidate to describe the long time asymptotics of the evolution model \eqref{KellerS} for general initial data, that
will be treated elsewhere. Finally, we point out that for the model \eqref{KellerSg} with $d\geq3$ and in the diffusion dominated case $m>2-2/d$, this asymptotic regime is shown in \cite[Theorem 5.6, Corollary 5.9]{Kim}, in the class of radially symmetric, continuous, and compactly supported initial data.

From the technical point of view, we cannot resort to classical concentration-compactness principles as used in \cite{Li1,Li2,CarBlanLau,Bed1}, which are very
related to homogeneous kernels as in \eqref{free}. Actually, we take advantage of the logarithmic interaction kernel to show the confinement of the density in
Section 2. This is the basic brick to show the existence of global minimizers that are radially decreasing due to symmetric decreasing rearrangement techniques.
We further identify them and show that they are compactly supported and smooth in their support in Section 3 using variational techniques. Finally, a non
standard application of the moving plane method in Section 4 shows that compactly supported stationary states of \eqref{KellerS} coincide with the unique up to
translation global minimizer of the previous sections.


\section{Minimization of the free energy functional}

Our goal is to minimize the functional $\mathsf{G}[\rho]$ given by \eqref{eq.1}
defined on $\mathcal{Y}_{M}$ for a given \emph{mass} $M>0$. Set
$\mathsf{G}[\rho]=\mathsf{H}[\rho]+\mathsf{W}[\rho]$ where
\begin{equation*}
\mathsf{H}[\rho]:=\frac{1}{m-1}\int_{\R^{2}}\rho^{m}dx
\end{equation*}
is the \emph{entropy functional}, defined on $L_{+}^{m}(\R^2)$, while
\begin{equation*}
\mathsf{W}[\rho]:=-\frac{1}{2}\int_{\R^2}(\mathcal{K}\ast\rho)(x)\rho(x)dx
\end{equation*}
is the \emph{interaction functional}, where
\begin{equation*}
\mathcal{K}(x):=-\frac{1}{2\pi}\log|x|\,.
\end{equation*}

Let us first check that
$\mathsf{W}[\rho]>-\infty$ in this class. Notice that for each
$\rho\in\mathcal{Y}_{M}$, H\"{o}lder's inequality implies that
\begin{align*}
\int_{\R^{2}}\int_{\R^{2}}\log^{-}|x-y|\,\rho(x)\rho(y)\,dx\,dy&=\int_{\R^{2}}\int_{|x-y|\leq
1}\left|\log|x-y|\right|\,\rho(x)\rho(y)\,dx\,dy\\&\leq\|\rho\|_{m}\int_{\R^{2}}\left(\int_{|x-y|\leq
1}|\log|x-y||^{m^{\prime}}dy\right)^{1/m^{\prime}}\rho(x)\,dx\\
&\leq C\,M \|\rho\|_{m}\nonumber,
\end{align*}
where $m^{\prime}=m/(m-1)$ and $C$ is a positive constant. Then,
we have that $\mathsf{G}[\rho]\in(-\infty,\infty]$ in
$\mathcal{Y}_{M}$. Let us define the class of radial densities as
\begin{equation*}
\mathcal{Y}^{rad}_{M}:=\left\{\rho\in L_{+}^{1}(\R^{2})\cap
L^{m}(\R^{2}):\,\|\rho\|_{1}=M,\rho=\rho^{\#}\right\}\,,
\end{equation*}
where $\rho^{\#}$ is the the \emph{spherical decreasing
rearrangement of} $\rho$, see for instance
\cite{Hardy,Bennett,Talenti} for the basic definitions and related
properties.

\begin{theorem}\label{varthm}
For any positive mass $M$, there exists a global radial minimizer
$\rho_{0}\in\mathcal{Y}^{rad}_{M}$ of the free energy functional
$\mathsf{G}$ in $\mathcal{Y}_{M}$. Moreover, global minimizers satisfy $\rho_0 \log
\rho_0 \in L^{1}(\R^{2})$.
\end{theorem}
\begin{proof}
We split the proof in three parts, proving first that global
minimizers must be radial, and thus we can restrict our study to
$\mathcal{Y}^{rad}_{M}$. We next show that the functional
$\mathsf{G}$ is bounded from below in $\mathcal{Y}^{rad}_{M}$, and
finally that the infimum is achieved in $\mathcal{Y}^{rad}_{M}$.

\par\noindent\emph{Step 1.~ The candidates to be global minimizers of $\mathsf{G}$ are radial.\/}\\
As soon as the interaction term $\mathsf{W}[\rho]$ is finite, the
interaction functional $\mathsf{W}[\rho]$ decreases under
rearrangement as proven in \cite[Lemma 2]{CarlenLoss}, in the
sense that
\begin{equation}
\int_{\R^{2}}\int_{\R^{2}}\log|x-y|\,\rho(x)\rho(y)\,dx\,dy\geq
\int_{\R^{2}}\int_{\R^{2}}\log|x-y|\,\rho^{\#}(x)\rho^{\#}(y)\,dx\,dy\,.
\label{logineqr}
\end{equation}
This shows that
\[
\inf_{\mathcal{Y}_{M}} \mathsf{G}=\inf_{\mathcal{Y}^{rad}_{M}}
\mathsf{G}\,.
\]
Actually, all the minimizers of $\mathsf{G}$ in the class
$\mathcal{Y}_{M}$ must be radially decreasing, i.e. they lay in
the class $\mathcal{Y}_{M}^{rad}$. Indeed, if $\rho$ is a global
minimizer in $\mathcal{Y}_{M}$, by inequality \eqref{logineqr} we
have that $\rho^{\#}$ is a radially decreasing global minimizer of
$\mathsf{G}$. Since the $L^m$-norms of $\rho$ and $\rho^{\#}$ are
equal, from $\mathsf{G}[\rho]=\mathsf{G}[\rho^{\#}]$ we deduce
that
\[
\int_{\R^{2}}\int_{\R^{2}}\log|x-y|\,\rho(x)\rho(y)\,dx\,dy=\int_{\R^{2}}\int_{\R^{2}}\log|x-y|\,\rho^{\#}(x)\rho^{\#}(y)\,dx\,dy\,.
\]
Hence, using \cite[Lemma 2]{CarlenLoss} again, we find that $\rho$
must be a translation of $\rho^{\#}$, that is
$\rho(x)=\rho^{\#}(x-y)$ for some $y\in \R^{2}$. Moreover, we have
\[
\int_{\R^2}x\rho(x)dx=yM+\int_{\ren}x\rho^{\#}(x)dx=yM,
\]
and thus the zero center-of-mass condition holds if and only if
$y=0$, giving $\rho=\rho^{\#}$, namely $\rho$ is radially
decreasing.

\par\noindent\emph{Step 2.~$\mathsf{G}$ is bounded from below in $\mathcal{Y}^{rad}_{M}$.\/  }
Here, we follow arguments from \cite{CalCar06}. For any
$\rho\in\mathcal{Y}_{M}$ such that
$$
\rho\,log\rho\,,\quad \rho\,\log(1+|x|^{2})\in L^1(\R^{2}),
$$
the \emph{logarithmic Hardy-Littlewood-Sobolev inequality}
\cite{CarlenLoss} implies that there exists a constant $C(M)>0$
such that
\begin{equation}
\int_{\R^{2}}\rho\log\rho\,dx\geq-\frac{2}{M}\int_{\R^{2}}\int_{\R^{2}}\log|x-y|\,\rho(x)\rho(y)\,dx\,dy-C(M)\,.
\label{eq.2}
\end{equation}

Let us start by showing a bound from below in a restricted class
of densities. Consider $\rho$ in $\mathcal{Y}^{rad}_{M}$ and first
assume that $\rho$ is continuous, with compact support. Applying
\eqref{eq.2} to \eqref{eq.1} we have
\begin{equation}
\mathsf{G}[\rho]\geq-\frac{M}{8\pi}C(M)+\int_{\R^{2}}\left(\frac{\rho^{m}}{m-1}-\frac{M}{8\pi}\,\rho\,\log\rho\right)dx\label{eq.3}.
\end{equation}
Now, let us choose $\theta\in(0,1)$ and a value $\kappa_{\,
\theta,m}>1$ such that
$$
\theta\frac{r^{m}}{m-1}-\frac{M}{8\pi}\,r\,\log r>0\quad\forall
r>\kappa_{\theta,m}.
$$
Then, we deduce
\begin{align*}
\int_{\R^2}\rho\,\log\rho\,dx&=\int_{\rho\leq1}\rho\,\log\rho\,dx+\int_{1<\rho\leq\kappa_{\theta,m}}\rho\,\log\rho\,dx+\int_{\rho>\kappa_{\theta,m}}\rho\,\log\rho\,dx\\
&\leq M\log\kappa_{\theta,m}+\frac{8\pi}{M(m-1)}\,\theta\int_{\R^{2}}\rho^{m}dx,\label{LlogL}
\end{align*}
and therefore
\[
\int_{\R^{2}}\left(\frac{\rho^{m}}{m-1}-\frac{M}{8\pi}\,\rho\,\log\rho\right)dx\geq-\frac{M^2}{8\pi}\log\kappa_{\theta,m}+\frac{1-\theta}{m-1}\int_{\R^{2}}\rho^{m}dx\,.
\]
Hence, we infer from \eqref{eq.3} that
\begin{equation}
\mathsf{G}[\rho]\geq-\frac{M}{8\pi}\left(M\log\kappa_{\theta,m}+C(M)\right)+\frac{1-\theta}{m-1}\int_{\R^{2}}\rho^{m}dx\label{G}.
\end{equation}

This bound from below being only dependent on the $L^m$-norm can
be extended to $\mathcal{Y}^{rad}_{M}$ by a density argument that
we detail next. If $\rho\in \mathcal{Y}^{rad}_{M}$ is less
regular, let us take a nondecreasing sequence of radially
decreasing, compactly supported, continuous nonnegative functions
$\widetilde{\rho}_{n}$ converging strongly to $\rho$ in $L^{1}(\R^{2})\cap
L^{m}(\R^2)$: such choice is always possible, since we can
approximate $\rho$ first by smooth functions, then the sequence of
their rearrangements satisfies the required conditions by the
$L^p$-contraction property of the rearrangement map. If
$\|\widetilde{\rho}_{n}\|_{1}=M_{n}$, let us construct the
sequence
\[
\rho_{n}:=\frac{M}{M_{n}}\,\widetilde{\rho}_{n}.
\]
Thus $\rho_{n}\in \mathcal{Y}^{rad}_{M}$. Besides, from
$M_{n}\nearrow M$, we get $\rho_{n}\nearrow\rho$ strongly in
$L^{1}(\R^{2})\cap L^{m}(\R^2)$, $\|\rho_{n}\|_{1}=M$, and we can apply
inequality \eqref{eq.3} to deduce
\begin{equation}
\mathsf{G}(\rho_{n})\geq-\frac{M}{8\pi}\left(M\log\kappa_{\theta,m}+C(M)\right)+\frac{1-\theta}{m-1}\int_{\R^{2}}\rho_{n}^{m}dx\,.
\label{Gn}
\end{equation}
H\"{o}lder's inequality implies
\begin{align*}
&\left|\int_{\R^{2}}\int_{|x-y|\leq1}\log|x-y|\,\left(\rho_{n}(x)\rho_{n}(y)-\rho(x)\rho(y)\right)\,dx\,dy\right|\\
&\qquad=\left|\int_{\R^{2}}\int_{|x-y|\leq1}\log|x-y|\left[(\rho_{n}(x)-\rho(x))\rho_{n}(y)+(\rho_{n}(y)-\rho(y))\rho(x)\right]\,dx\,dy\right|\\
&\qquad\leq\|\rho_{n}-\rho\|_{m}\int_{\R^{2}}\left[\left(\int_{|x-y|\leq1}|\log|x-y||^{m^{\prime}}\,dx\right)^{1/m^{\prime}}\rho_{n}(y)\right]dy\\
&\qquad\,\,\,\,\,\,+\|\rho_{n}-\rho\|_{m}\int_{\R^{2}}\left[\left(\int_{|x-y|\leq1}|\log|x-y||^{m^{\prime}}\,dy\right)^{1/m^{\prime}}\rho(x)\right]dx\\
&\qquad=2C\,M\|\rho_{n}-\rho\|_{m}\rightarrow0\,.
\end{align*}
Concerning the positive part of $\log|x-y|$, since $\rho_n$ is a
nondecreasing sequence converging to $\rho$, we have by the
monotone convergence theorem that
\[
\int_{\R^{2}}\int_{|x-y|>1}\log^{+}|x-y|\,\rho_{n}(x)\rho_{n}(y)\,dx\,dy\nearrow\int_{\R^{2}}\int_{|x-y|>1}\log^{+}|x-y|\,\rho(x)\rho(y)\,dx\,dy\,,
\]
as $n\rightarrow\infty$, and thus $\mathsf{H}(\rho_n)\rightarrow
\mathsf{H}(\rho)$, $\mathsf{W}(\rho_n)\rightarrow
\mathsf{W}(\rho)$ as $n\rightarrow\infty$. Hence, we can pass to
the limit in \eqref{Gn} and get \eqref{G} in
$\mathcal{Y}^{rad}_{M}$.

\par\noindent\emph{Step 3.~The infimum of $\mathsf{G}$ is achieved in $\mathcal{Y}^{rad}_{M}$.\/}\\
Let
\begin{equation*}
\mathcal{I}:=\inf_{\rho\in \mathcal{Y}^{rad}_{M}}\mathsf{G}(\rho)
\end{equation*}
and let us choose a minimizing sequence of $\mathsf{G}$,
\emph{i.e.} a sequence $\left\{\rho_{n}\right\}_{n\in \N}$ in
$\mathcal{Y}^{rad}_{M}$ such that
\begin{equation}
\mathsf{G}[\rho_{n}]\rightarrow\mathcal{I}\quad\,\text{as }n\rightarrow\infty.\label{eq.4}
\end{equation}
By the control of the functional $\mathsf{G}$ in \eqref{G}, we get
that $\left\{\rho_{n}\right\}_{n\in\N}$ is bounded in
$L^{m}(\R^2)$ hence by \eqref{eq.4} it follows that
$\left\{\mathsf{W}[\rho_{n}]\right\}_{n\in\N}$ is bounded. In
order to control the behavior at infinity, we follow similar
arguments as in \cite{McCann97} and \cite[Proposition
7.10]{TheseCalvez}. For all $R\geq1$ and any $\rho\in
L_{+}^{1}(\R^{2})\cap L^{m}(\R^{2})$, define the functional
\begin{equation}
\mathsf{W_{R}}[\rho]:=\int_{\R^{2}}\int_{|x-y|>R}\log|x-y|\,\rho(x)\rho(y)\,dx\,dy.\label{W_R}
\end{equation}
By H\"{o}lder's inequality, we have
\begin{align}
&\mathsf{W_{R}}[\rho]=\int_{\R^{2}}\int_{|x-y|>R}\log|x-y|\,\rho(x)\rho(y)\,dx\,dy\label{eq.5}\\
&=\int_{\R^{2}}\int_{\R^{2}}\log|x-y|\,\rho(x)\rho(y)\,dx\,dy+\int_{\R^{2}}\int_{|x-y|\leq
R}|\log|x-y||\,\rho(x)\rho(y)\,dx\,dy\nonumber\\
&\leq \mathsf{W}[\rho]+\|\rho\|_{m}\int_{\R^{2}}\left(\int_{|x-y|\leq R}|\log|x-y||^{m^{\prime}}dy\right)^{1/m^{\prime}}\rho(x)\,dx\leq \mathsf{W}[\rho]+C\,M
\|\rho\|_{m}\nonumber,
\end{align}
where $m^{\prime}=m/(m-1)$. In particular, by \eqref{eq.5} it
follows that $\left\{\mathsf{W}_{R}[\rho_{n}]\right\}_{n\in\N}$ is
bounded. Now, let $x\in\R^{2}$ with $|x|\geq1$ and notice that $
\left\{y\in\R^{2}:\,x\cdot
y\leq0\right\}\subset\left\{y\in\R^{2}:\,|x-y|\geq1\right\}. $
Then, since  $\rho$ is nonnegative, for all $R\geq1$ we get
\begin{equation}
\mathsf{W_{1}}[\rho]\geq\int_{|x|>R}\int_{x\cdot y\leq0}\log|x-y|\rho(x)\rho(y)\,dx\,dy\label{eq.29}.
\end{equation}
Since $x\cdot y\leq0$ implies $|x-y|\geq|x|$, we infer from \eqref{eq.29} that
\begin{align}
&\mathsf{W_{1}}[\rho]\geq\int_{|x|>R}\int_{x\cdot y\leq0}\log|x|\rho(x)\rho(y)\,dx\,dy\label{eq.31},
\end{align}
then if we assume $\rho=\rho^{\#}$, we find
\begin{equation}
\mathsf W_{1}[\rho]\geq\log R\,\int_{|x|>R}\int_{x\cdot
y\leq0}\rho(x)\rho(y)\,dx\,dy=\frac{M\log
R}{2}\int_{|x|>R}\rho(x)\,dx. \label{eq.6}\end{equation} Thus the
fact that $\left\{\mathsf{W}[\rho_{n}]\right\}_{n\in\N}$ is
bounded and \eqref{eq.6} yield
\begin{equation}
\sup_{n\in\N}\int_{|x|>R}\rho_{n}(x)\,dx\leq \frac{C}{\log R}\underset{R\rightarrow\infty}{\longrightarrow0}\label{eq.7}
\end{equation}
that is the so called \emph{confinement of the mass}. In order to
check that $\left\{\rho_{n}\right\}$ is locally
\emph{equi-integrable}, we just observe that for given
$\varepsilon>0$, setting
\begin{equation}\label{eq.72}
\mathsf{a}:=\sup_{n\in\N}\|\rho_{n}\|_{m}<\infty
\end{equation}
for any subset $A$ of $\R^{N}$ such that
$|A|<\delta:=(\varepsilon/\mathsf{a})^{m^{\prime}}$ we have, by
H\"{o}lder's inequality
\[
\int_{A}\rho_{n}\,dx\leq\mathsf{a}\,|A|^{\frac{m-1}{m}}<\varepsilon.
\]
for all $n\in\N$, that is the sequence
$\left\{\rho_n\right\}_{n\in\N}$ is equi-integrable. According to
Dunford-Pettis theorem using \eqref{eq.7} and \eqref{eq.72}, there
exists a function $\rho_{0}\in L^{1}_{+}(\R^{2})\cap
L^{m}(\R^{2})$ such that (up to subsequence)
\begin{equation}
\rho_{n}\rightharpoonup \rho_{0}\quad\text{weakly in }L^{1}(\R^{2})\cap L^{m}(\R^{2})\label{eq.8}
\end{equation}
and $\|\rho_0\|_{1}=M$. Furthermore
\[
\|\rho_{0}\|_{m}\leq\liminf_{n\rightarrow\infty}\|\rho_{n}\|_{m}\leq C.
\]
In particular, the interaction energy $\mathsf{W}[\rho_{0}]$ of
$\rho_0$ is bounded from below because the functional $\mathsf{G}$
is. Our aim is now to show that $\mathsf{W}$ is lower
semicontinuous with respect to the $L^{1}\cap L^{m}-$ weak
convergence, taking advantage of some arguments shown in
\cite{BlaCaCarSiam}. Then fix $\varepsilon\in(0,1)$, $R>1$ and
write
\[
\mathsf{W}[\rho_{n}]=\mathsf{A^{\varepsilon}}[\rho_{n}]+\mathsf{B^{\varepsilon}}[\rho_{n}]+\mathsf{W_{R}}[\rho_{n}]
\]
where
\begin{equation*}
\left.
\begin{array}
[c]{l}%
\mathsf{A^{\varepsilon}}[\rho]:={\displaystyle\int_{\R^{2}}\int_{|x-y|\leq\varepsilon}\log|x-y|\,\rho(x)\rho(y)\,dx\,dy}\\
\\
\mathsf{B^{\varepsilon}}[\rho]:={\displaystyle\int_{\R^{2}}\int_{\varepsilon<|x-y|\leq R}\log|x-y|\,\rho(x)\rho(y)\,dx\,dy}\\
\\
\end{array}
\right.
\end{equation*}
and the functional $\mathsf{W}_{R}$ is defined in \eqref{W_R}.
We notice that the same arguments used to prove inequality \eqref{eq.5} yield
\[
\mathsf{A^{\varepsilon}}[\rho_{n}]\leq CM\|\rho_n\|_{m}\left(\int_{0}^{\varepsilon}r|\log r|^{m^{\prime}}dr\right)^{1/m^{\prime}}
\]
then
\begin{equation}
\mathsf{A^{\varepsilon}}[\rho_{n}]\rightarrow0\quad\text{as }\varepsilon\rightarrow0,\text{ uniformly in }n.\label{A}
\end{equation}
Observe that we can use the equi-integrability of the sequence
$\left\{\rho_{n}\right\}$ and the fact that
$\rho_{n}\rightharpoonup\rho_{0}$ weakly in $L_{+}^{1}(\R^{2})$ to
apply Lemma 2.3 in \cite{BlaCaCarSiam} and find
\begin{equation}
\rho_{n}\otimes\rho_{n}\rightharpoonup\rho_{0}\otimes\rho_{0}\quad\text{weakly in }L_{+}^{1}(\R^{2}\times\R^{2}).\label{eq.30}
\end{equation}
Then, since the function $\log|x-y|$ is bounded in $\left\{\varepsilon<|x-y|\leq R\right\}$ we have that
\begin{equation}
\mathsf{B^{\varepsilon}}[\rho_{n}]\rightarrow{\displaystyle\int_{\R^{2}}\int_{\varepsilon<|x-y|\leq
R}\log|x-y|\,\rho_0(x)\rho_0(y)\,dx\,dy}\quad\text{as }n\rightarrow\infty\label{B}
\end{equation}
It remains then to get a bound from below of the last integral
$\mathsf{W_{R}}[\rho_{n}]$, for large $n$. In order to do this, we
first point out that \eqref{eq.30} implies that the sequence of
densities $(\rho_{n}\otimes\rho_{n})(x,y)$ converges to
$(\rho_{0}\otimes\rho_{0})(x,y)$ in the weak-$*$ sense as
measures. Then using the fact that the function $\log^{+}|x-y|$ is
bounded from below and obviously lower semicontinuous in the set
$\left\{(x,y):x\in\R^{2},|x-y|>R\right\}$, inequality 5.1.15 in
\cite{AmbrGi} gives
\begin{equation}
\liminf_{n\rightarrow\infty}\mathsf{W_{R}}[\rho_{n}]\geq{\displaystyle\int_{\R^{2}}\int_{|x-y|>R}\log|x-y|\,\rho_0(x)\rho_0(y)\,dx\,dy}.\label{C}
\end{equation}
In particular, combining this last inequality with \eqref{eq.5} we
derive that $$(\log|\cdot|\ast\rho_{0})\rho_{0}\in L^{1}(\R^{2})$$
and $W[\rho_0]$ is finite. Now, using \eqref{B} and \eqref{C} we
get
\begin{multline*}
\liminf_{n\rightarrow\infty}\mathsf{W}[\rho_n]\geq\liminf_{n\rightarrow\infty}\mathsf{A^{\varepsilon}}[\rho_{n}]\\+{\displaystyle\int_{\R^{2}}\int_{\varepsilon<|x-y|\leq
R}\log|x-y|\,\rho_0(x)\rho_0(y)\,dx\,dy}\\+{\displaystyle\int_{\R^{2}}\int_{|x-y|>R}\log|x-y|\,\rho_0(x)\rho_0(y)\,dx\,dy}
\end{multline*}
thus letting $\varepsilon\rightarrow0$, property \eqref{A} implies
\[
\liminf_{n\rightarrow\infty}\mathsf{W}[\rho_n]\geq\mathsf{W}[\rho_0].
\]
Using \eqref{eq.8}, this gives in turn
\begin{equation}
\mathcal{I}=\liminf_{n\rightarrow\infty}\mathsf{G}[\rho_n]\geq\mathsf{H}[\rho_0]+\mathsf{W}[\rho_0].\label{minim}
\end{equation}
By taking the rearrangement $\rho_{0}^{\#}$ of $\rho_{0}$, since
as $\mathsf{W}[\rho_0]$ is finite, inequality \eqref{logineqr}
implies
\begin{equation}
\mathsf{W}[\rho_0]\geq \mathsf{W}[\rho^{\#}_0]\,,\label{eq.32}
\end{equation}
hence \eqref{minim} gives
\[
\mathcal{I}\geq\mathsf{G}[\rho^{\#}_0].
\]
With this we have finished the proof of existence of global radial
minimizers. Finally, we notice by \eqref{eq.32} that
$\mathsf{W}_{1}[\rho^{\#}_0]$ is finite, so for all $R\geq1$
inequality \eqref{eq.31} provides
\begin{align*}
&\int_{|x|>R}\int_{x\cdot y\leq0}\log|x|\rho_{0}^{\#}(x)\rho_{0}^{\#}(y)dx dy
\leq\mathsf{W}_{1}[\rho^{\#}_{0}]<\infty
\end{align*}
that is
\[
\frac{M}{2}\int_{|x|>R}\log|x|\,\rho_{0}^{\#}(x)dx\leq C
\]
namely $\rho^{\#}_{0}\log(1+|x|^{2})\in L^1({\R^2})$.
\end{proof}

\begin{remark}
Let us point out the the previous proof works in any dimension
since the logarithmic HLS inequality holds true with a constant
that depends only on the dimension and the mass. We also emphasize
that the use of the logarithmic potential is crucial here, since we do
not know how to prove a quantitative confinement property when the Newtonian
potential for dimensions larger than two is used instead.
\end{remark}


\section{Identification, regularity, and uniqueness of global minimizers}
Our aim is to show a full characterization of any  minimizer
$\rho_0$ of the functional $\mathsf{G}$ to relate them to the
steady states to the 2D Keller-Segel model. We first deduce the
Euler-Lagrange conditions satisfied by critical points of the
functional.

\begin{theorem}\label{idenminimth}
Let $\rho_0\in\mathcal{Y_{M}}$ be a global minimizer of the free
energy functional $\mathsf{G}$ defined in \eqref{eq.1}. Then
$\rho_0$ satisfies
\begin{equation}
\frac{m}{m-1}\rho_{0}^{m-1}-\mathcal{K}\ast\rho_{0}=
\mathsf{D}[\rho_{0}]\quad a.e. \text{ in }
\mbox{supp}(\rho_{0})\label{eq.33}
\end{equation}
and
\begin{equation}
\frac{m}{m-1}\rho_{0}^{m-1}-\mathcal{K}\ast\rho_{0}\geq
\mathsf{D}[\rho_{0}]\quad a.e. \text{ outside }
\mbox{supp}(\rho_{0})\label{eq.34}
\end{equation}
where
\[
\mathsf{D}[\rho_{0}]=\frac{2}{M}\mathsf{G}[\rho_0]+\frac{m-2}{M(m-1)}\|\rho_{0}\|_{m}^{m}.
\]
As a consequence, any global minimizer of $\mathsf{G}$ verifies
\begin{equation}\label{eq.19}
\frac{m}{m-1}\rho_{0}^{m-1}=\left(\mathcal{K}\ast\rho_{0}+\mathcal{D}[\rho_{0}]\right)_{+}.
\end{equation}
\end{theorem}

\begin{proof}
The technical difficulty here is to make good variations of the
minimizer under the low available regularity conditions on
$\rho_0$ obtained from Theorem \ref{varthm}, namely $\rho_0 \in
L^{1}_+\cap L^{m}(\R^{2})$ and $\rho_0 \log \rho_0 \in
L^{1}(\R^{2})$. We use some ideas from \cite{Strohmer}. We first
show \eqref{eq.33}. Let $\rho_{0}$ be a radially decreasing
minimizer of $\mathsf{G}$. Taking any $\varepsilon>0$ and a test
function $\psi\in C_{0}^{\infty}(\R^{2})$ such that
$\psi(x)=\psi(-x)$, let us define the function
\[
\varphi(x)=\left(\psi(x)-\frac{1}{M}\int_{\R^{2}}\psi(x)\,\rho_{0}(x)\,dx\right)\rho_{0}(x).
\]
We point out that $\varphi\in L^{1}(\R^{2})\cap L^{m}(\R^{2})$, besides
\begin{equation}
\int_{\R^{2}}\varphi(x)\,dx=\int_{\R^{2}}x\varphi(x)\,dx=0,\label{eq.35}
\end{equation}
and $\text{supp}(\varphi)\subseteq\text{supp}(\rho_{0})=:E$.
Moreover, for
$\varepsilon<\varepsilon_{0}:=(2\|\psi\|_{\infty})^{-1}$ we find
\begin{equation*}
\rho_{0}+\varepsilon\varphi=\rho_{0}(x)\left(1+\varepsilon\left(\psi-\frac{1}{M}\int_{\R^{2}}\psi\,\rho_{0}\,dx\right)\right)
\geq\rho_{0}(x)(1-2\varepsilon\|\psi\|_{\infty})\nonumber\geq0.
\end{equation*}
Due to \eqref{eq.35}, we have  $\rho_{0}+\varepsilon\varphi\in
\mathcal{Y}_{M}$, thus we can calculate the first variation
$\frac{\delta \mathsf{G}}{\delta{\varphi}}(\rho_{0})$ of the
functional $\mathsf{G}$. Noting that
$\text{supp}(\rho_{0}+\varepsilon\varphi)\subseteq E$, we get
\begin{align}
\frac{\mathsf{G}[\rho_{0}+\varepsilon\varphi]-\mathsf{G}[\rho_{0}]}{\varepsilon}=\frac{1}{m-1}\int_{\mathring{E}}\frac{(\rho_{0}+\varepsilon\varphi)^{m}-\rho^{m}_{0}}{\varepsilon}dx
-\int_{\R^{2}}\mathcal{K}\ast\rho_{0}\varphi\,
dx+\varepsilon\mathsf{W}[\varphi]\label{eq.17}.
\end{align}
Using the first order Taylor expansion of
$(\rho_{0}+\varepsilon\varphi)^{m}$ at $\varepsilon=0$, we have
\[
\int_{\mathring{E}}\frac{(\rho_{0}+\varepsilon\varphi)^{m}-\rho^{m}_{0}}{\varepsilon}dx=m\int_{0}^{1}\mathcal{G}_{\varepsilon}(t)dt
\]
where
\[
\mathcal{G}_{\varepsilon}(t):=\int_{\mathring{E}}|\rho_{0}+\varepsilon t\varphi|^{m-2}(\rho_{0}+\varepsilon t\varphi)\,\varphi\,dx.
\]
By the definition of $\mathcal{G}_{\varepsilon}(t)$, it is obvious
that for all $t\in[0,1]$ and $\varepsilon<\varepsilon_{0}$, one
has
\[
|\mathcal{G}_{\varepsilon}(t)|\leq(\|\rho_{0}\|_{m}+\varepsilon_{0}\|\varphi\|_{m})^{m-1}\|\varphi\|_{m}\in
L^{1}_{t}(0,1)\,,
\]
then Lebesgue's dominated convergence yields
\begin{equation}
\int_{\mathring{E_{0}}}\frac{(\rho_{0}+\varepsilon\varphi)^{m}-\rho^{m}_{0}}{\varepsilon}dx\xrightarrow[]{\varepsilon\rightarrow 0}
m\int_{\R^{2}}\rho_{0}^{m-1}\varphi\,
dx.\label{eq.18}
\end{equation}
In addition, as $\rho_{0}(x)\log(1+|x|^{2})\in L^1(\R^2)$, the
algebraic inequality
$$
\log|x-y|\leq\frac{1}{2}(\log 2+\log(1+|x|^{2}))+\log(1+|y|^{2})
$$
and the estimate $|\varphi(x)|\leq2 \|\psi\|_{\infty}\rho_{0}(x)$
give $\mathsf{W}[\varphi]\leq C$. Therefore using this last
property and \eqref{eq.18} to pass to the limit in \eqref{eq.17}
as $\varepsilon\rightarrow0$, we obtain that
$$
\int_{\R^{2}}\varphi\mathcal{F}(\rho_{0})dx \geq 0\,,
$$
since
$\mathsf{G}[\rho_{0}+\varepsilon\varphi]\geq\mathsf{G}[\rho_{0}]$, where
\[
\mathcal{F}(\rho_{0}):=\frac{m}{m-1}\rho_{0}^{m-1}-\mathcal{K}\ast\rho_{0}=\frac{m}{m-1}\rho_{0}^{m-1}+\frac{1}{2\pi}\log|x|\ast\rho_{0}.
\]
Taking $-\psi$ instead of $\psi$, we finally obtain
\begin{equation}\label{first variation}
\int_{\R^{2}}\varphi\mathcal{F}(\rho_{0})dx = 0.
\end{equation}
By the definition of $\varphi$, we conclude that
\[
\int_{\R^{2}}\left[\mathcal{F}(\rho_{0})-\mathsf{D}[\rho_0]\,\right]\rho_{0}\,\psi\,dx=0\,,
\]
for all even functions $\psi\in C_{0}^{\infty}(\R^{2})$. Hence, we
deduce
\begin{equation}
\frac{m}{m-1}\rho_{0}^{m-1}-\mathcal{K}\ast\rho_{0}=\mathsf{D}[\rho_0]\quad\text{a.e. in }
\left\{\rho_{0}>0\right\}. \label{eq.11}
\end{equation}
Now, we turn to the proof of \eqref{eq.34}. Let us take an even
function $\psi\in C_{0}^{\infty}(\R^{2})$ with $\psi\geq 0$ such
that $\psi(x)\in [0,1]$, let us define the function
\[
\varphi=\psi-\frac{\rho_{0}}{M}\int_{\R^{2}}\psi\,dx.
\]
Then $\varphi\in L^{1}(\R^2)\cap L^{m}(\R^2)$ and
\[
\int_{\R^{2}}\varphi(x)\,dx=\int_{\R^{2}}x\varphi(x)\,dx=0.
\]
In addition, denoting by $|\cdot|_{N}$ the $N-$ dimensional Lebesgue measure, we have
\begin{equation*}
\rho_{0}+\varepsilon\varphi\geq\rho_{0}\left(1-\frac{\varepsilon}{M}\int_{\R^{2}}\psi dx\right)\geq\left(1-\frac{\varepsilon}{M}|\text{supp
}\psi|_{N}\right)\rho_{0}(x),
\end{equation*}
then $\rho_{0}+\varepsilon\varphi\geq0$ for small $\varepsilon$ in
$\mbox{supp}(\rho_0)$ and outside since $\psi\geq 0$, hence
$\rho_{0}+\varepsilon\varphi\in\mathcal{Y}_{M}$. Arguing as
before, we obtain from \eqref{first variation}
\[
\int_{\R^{2}}\left[\mathcal{F}(\rho_{0})-\mathsf{D}[\rho_0]\right]\psi\, dx\geq0
\]
for all the functions $\psi$ chosen as above, implying
\[
\frac{m}{m-1}\rho_{0}^{m-1}\geq\mathcal{K}\ast\rho_{0}+\mathsf{D}[\rho_0]
\quad a.e. \mbox{ outside supp}(\rho_0).
\]
\end{proof}

\begin{remark}
Let us point out that inequality \eqref{eq.34} is a consequence of the
positivity and mass constraints on the class of possible minimizers, i.e,
due to the fact that we are working with a optimization problem
with convex constraints.
\end{remark}

Actually, we can show many properties about the regularity of
global radial minimizers to the free energy functional
$\mathsf{G}$. Now, we give some information concerning the
\emph{asymptotic behavior} at infinity of the \emph{logarithmic
potential} of any density $\rho_{0}\in\mathcal{Y}_{M}$, namely the
Newtonian potential
\begin{equation}\label{eq.12}
u(x):=(\mathcal{K}\ast\rho_{0})(x)=-\frac{1}{2\pi}\int_{\R^{2}}\log|x-y|\,\rho_{0}(y)dy.
\end{equation}
The proof of the following result is contained in \cite[Lemma
1.1]{CT}. Let $\rho\in\mathcal{Y}_{M}$. Then we have
\begin{equation}\label{asymptbehth1}
\lim_{|x|\to\infty}\frac{u(x)}{\mathcal{K}(x)}= M\,.
\end{equation}
With this further result in hand, we are now ready to give more
information about the regularity of the radially decreasing
minimizers of $\mathsf{G}$.

\begin{theorem}\label{regulth1}
All radially decreasing global minimizers of $\mathsf{G}$ in
$\mathcal{Y}_{M}$ are compactly supported continuous functions in
$\R^2$ and smooth inside their support.
\end{theorem}
\begin{proof}
Let $\rho_0$ be a radially decreasing minimizer of $\mathsf{G}$.
Then there is a ball $B_{R}(0)$ such that
$\left\{\rho_{0}>0\right\}=B_{R}(0)$. Let us consider the
logarithmic potential of $\rho_{0}$, namely the function $u$
defined in \eqref{eq.12}. As $\rho_{0}\in L^{m}(\R^2)$, by
\cite[Lemma 9.9]{Gilbarg} we have $u\in W_{loc}^{2,m}(\R^{2})$. By
Morrey's Theorem ($m>1$), it follows that $u\in
L_{loc}^{\infty}(\R^{2})$, and by equation \eqref{eq.11} we get
\begin{equation}
\frac{m}{m-1}\rho_{0}^{m-1}=u+C \text{ a.e. in } B_{R}(0).\label{stat}
\end{equation}
Thus from the monotonicity of $\rho_0$ we deduce $\rho_{0}\in
L^{\infty}(\R^{2})$. Hence \cite[ Lemma 4.1]{Gilbarg} implies
$u\in C^{1}(\R^{2})$. Now, for all $r>0$ we define the \emph{mass
function of} $\rho_0$
\[
M_{\rho_{0}}(r)=\int_{B_{r}(0)}\rho_{0}(x)\,dx.
\]
Take any $R_{1}<R$ and consider the following boundary value problem
\begin{equation}
\left\{
\begin{array}
[c]{l}%
-\Delta v=\rho_{0}\quad\text{ in } B_{R_{1}}(0)
\\
\\
v(x)=u(R_{1})\quad\text{ on }\partial B_{R_{1}}(0)
\end{array}
\right.\label{eq.13}
\end{equation}
The logarithmic potential \eqref{eq.12} solves problem
\eqref{eq.13}, whence $u=v$ on $\overline{B_{R_{1}}(0)}$. On the
other hand, the solution of \eqref{eq.13} can be written as in
\cite{Talenti,Bennett}: if $r=|x|\in(0,R_{1})$,
\[
u(r)-u(R_1)=v(r)-v(R_1)=\frac{1}{4\pi}\int_{\pi r^{2}}^{\pi R_{1}^{2}}\frac1s\int_{0}^{s}\rho_{0}^{\ast}(\sigma)d\sigma\,ds
\]
where $\rho_{0}^{\ast}$ is the \emph{one dimensional decreasing
rearrangement of} $\rho_{0}$. Differentiating we get
\begin{equation*}
u^{\prime}(r)=-\frac{1}{2\pi r}\int_{0}^{\pi r^{2}}\rho_{0}^{\ast}(\sigma)d\sigma=-\frac{1}{2\pi r}\int_{B_{0}(r)}\rho_{0}(x)\,dx=-\frac{M_{\rho_{0}}(r)}{2\pi
r},
\end{equation*}
that is
\begin{equation}
\frac{d}{dr}(\rho_{0}\ast \mathcal{K})(r)=-\frac{M_{\rho_{0}}(r)}{2\pi r}\quad r>0.
\label{eq.14st}
\end{equation}
By identity \eqref{eq.14st} it follows that $\rho_{0}$ is
\emph{smooth} inside its support. Indeed, following some arguments
of \cite{Kim}, first we observe that the function
\[
f(r):=-\frac{M_{\rho_{0}}(r)}{2\pi r}
\]
is continuous for $r>0$ and $f(r)\rightarrow0$ as $r\rightarrow0$: indeed, we have
\[
\lim_{r\rightarrow0}f(r)=-\lim_{r\rightarrow0}\frac{1}{r}\int_{0}^{r}t\,\rho_{0}(t)\,dt=-\lim_{r\rightarrow0}r\,\rho_{0}(r)=0.
\]
Thus $u=\mathcal{K}\ast\rho_{0}$ is differentiable everywhere in
the positive set $\left\{\rho_{0}>0\right\}=B_{R}(0)$ of
$\rho_{0}$. This property and \eqref{stat} imply that $\rho_{0}$
is differentiable in $B_{R}(0)$, so $f(r)$ is twice
differentiable. Then we can repeat this argument and conclude.

Let us prove that $\rho_{0}$ has \emph{compact support}. There are
two different ways to prove this property: one is based on the
asymptotic behavior of the log-potential, the other relies on a
pure ODE approach relating our global minimizers to nonlinear
elliptic equations. We show both methods since they give
complementary information. Concerning the first one, we simply use
\eqref{asymptbehth1} to infer that $u(x)\sim
M\mathcal{K}(x)\rightarrow-\infty$ as $|x|\rightarrow\infty$,
hence \emph{if} equation \eqref{stat} were satisfied for all $x$,
for a sufficiently large $R$ we would have $\rho_{0}<0$ for all
$|x|>R$, which is a contradiction. Then $\rho_{0}$ must have
compact support.

The other argument to prove that $\rho_{0}$ is compactly supported
is the following. By contradiction, let us suppose that
$\text{supp}(\rho_{0})=\R^{2}$. Then \eqref{eq.11} implies that
the function $\theta:=\rho_{0}^{m-1}\in L^{\frac{m}{m-1}}$ solves
the problem
\begin{equation}
\left\{
\begin{array}
[c]{l}%
-\Delta\theta=\dfrac{m-1}{m}\,\theta^{1/(m-1)}\quad\text{ in } \R^{2}
\\
\\
\theta\rightarrow0\quad\text{ as }|x|\rightarrow\infty.
\end{array}
\right.\label{eq.14}
\end{equation}
Since $\theta$ is radial, the first equation in \eqref{eq.14} (which is an \emph{Emden Fowler-type} equation) can be rewritten as
\begin{equation*}
-(r\,\theta^{\prime})^{\prime}=\frac{m-1}{m}\,r\,\theta^{1/(m-1)}\quad,r>0
\end{equation*}
and with the change of variables $r=e^{t}\,,w(t)=\theta(e^{t})$,
the same equation reads
\begin{equation}
w^{\prime\prime}(t)+\frac{m-1}{m}e^{2t}w(t)^{\frac{1}{m-1}}=0.\label{eq.15}
\end{equation}
Now we can invoke \cite[Corollary 1.2]{ZUZ}: since for all $a>0 $
we have
\[
\int_{a}^{\infty}e^{2t}\,dt=\int_{a}^{\infty}t^{\frac{1}{m-1}}\,e^{2t}\,dt=+\infty\,.
\]
We obtain that in both cases $m<2$, $m>2$ \emph{all} the proper
solution to \eqref{eq.15} are \emph{oscillatory}, namely they have
a sequence of zeros tending to $+\infty$. But this contradicts the
fact that $\theta$ is everywhere positive. The case $m=2$ is even
simpler. Indeed, in this case $\theta$ satisfies the linear
problem (recall that $\rho_{0}$ is smooth)
\[
-(r\,\theta^{\prime})^{\prime}=\frac{r}{2}\,\theta\quad,r>0
\]
and the condition $\theta\rightarrow0$ as $r\rightarrow\infty$ obliges $\theta$ to have the form
\[
\theta(r)=C\,J_{0}\left(\frac{r}{\sqrt{2}}\right)
\]
which is clearly oscillating, leading to contradiction. Therefore,
the support of $\rho_{0}$ must be compact. Finally, being the
Newtonian potential smooth together with \eqref{eq.19} implies
that the density $\rho_0$ is H\"older continuous in $\R^2$ with
exponent $1/(m-1)$.
\end{proof}

\begin{remark}
By equation \eqref{eq.19} and arguing as in {\rm\cite{CarBlanLau}}, we
have that $\theta:=\rho_{0}^{m-1}$ is the unique \emph{classical
solution} in $B(0,R)$, with zero boundary condition, to the
elliptic equation
\[
-\Delta\theta=\frac{m-1}{m}\,\theta^{1/(m-1)}.
\]
Therefore, we can write $\theta$ in terms of a scaling of the
solution $\zeta$ to the same problem in the unit ball, namely
\[
\rho_{0}(x)=R^{2(m-1)/(m-2)}\,\zeta\left(\frac{x}{R}\right).
\]
\end{remark}

With the above regularity of global minimizers, it is easy to show
that the distributional gradient in $\R^2$ of $\rho^m_0$ satisfies
$\nabla \rho^m_0 = \tfrac{m}{m-1}\rho_0
\nabla\rho_0^{m-1}=m\rho_0^{m-1}\left[\nabla \rho_0\right]_+$ with
the last gradient being the classical gradient in its support. As
a consequence,
\begin{equation}
\nabla\rho_0^m- \rho_0\nabla (\mathcal{K}\ast
\rho_0)=\rho_0\left[\nabla(\tfrac{m}{m-1}\rho_0^{m-1}-
\mathcal{K}\ast \rho_0)\right]_+=0\label{steadystdef}
\end{equation} in the sense of distributions.
We have deduced the following result:

\begin{corollary}\label{stationary}
Global minimizers of the free energy functional $\mathsf{G}$ are
stationary solutions of the two dimensional subcritical
Keller-Segel model \eqref{KellerS} in the distributional sense \eqref{steadystdef}.
\end{corollary}

Now, let us show the uniqueness of stationary states among the set of radially decreasing compactly supported  smooth inside their support solutions. As a
consequence, we conclude the uniqueness of global minimizers taking into account Corollary \ref{stationary} and Theorems \ref{varthm} and \ref{regulth1}. With
this aim, we briefly recall some of the main results contained in \cite{Kim}. We firstly start with the definition of mass concentration:
\begin{definition}
Let $\rho_{1},\rho_{2}\in L^{1}_{loc}(\R^{N})$, $N\geq1$, be two radially symmetric functions on $\R^{N}$. We say that $\rho_1$ is less concentrated than
$\rho_2$, and we write
$\rho_1\prec
\rho_2$ if for
all $r>0$ we get
\[
\int_{B_{r}(0)}\rho_1(x)dx\leq \int_{B_{r}(0)}\rho_2(x)dx.
\]
\end{definition}
The partial order relationship $\prec$ is called \emph{comparison of mass concentrations}.
Of course, this definition can be suitably adapted if $\rho_1,\rho_2$ are radially symmetric and locally integrable functions on a ball $B_{R}$. Besides, if
$\rho_1$
and $\rho_2$ are locally integrable on a general open set $\Omega$, we say that $\rho_1$ is less concentrated than $\rho_2$ and we write again $\rho_1\prec
\rho_2$ simply if
$\rho_1^{\#}\prec \rho_2^{\#}$.\smallskip\newline
If $\rho(x,t)$ is a locally integrable function on $\R^N$ for all times $t\geq0$, we define the time dependant \emph{mass function} of $\rho$ as
\begin{equation}
M_{\rho}(r,t)=\int_{B_r(0)}\rho(x,t)dx.\label{massfunct}
\end{equation}
If $\rho(x,t)$ is the solution to the evolution problem \eqref{KellerS}
where the initial data $\rho(x,0)$ is a continuous, compactly supported, radially decreasing function, then it is easy to check, see \cite{Kim}, that the mass
function $M_{\rho}(r,t)$ satisfies, in the support $\left\{x:|x|<R(t)\right\}$ of $\rho(\cdot,t)$ the PDE
\begin{equation}
\frac{\partial M_{\rho}}{\partial t}(r,t)=2\pi r\,\frac{\partial}{\partial r}\left(\left(\frac{1}{2\pi r}\frac{\partial M_{\rho}}{\partial r}\right)^{m}\right)+
\left(\frac{1}{2\pi r}\frac{\partial M_{\rho}}{\partial r}\right)M_\rho.\label{eqmass}
\end{equation}
Let us take a function $\rho(x,t)$ being $C^{1}$ in its positive set and $\rho(\cdot,t)\in L^{1}(\R^2)\cap L^{\infty}(\R^2)$ for each $t\geq 0$. We will say
that
$\rho$ is a subsolution (resp. a supersolution) to \eqref{eqmass} if the sign $\leq$ (respectively the sign $\geq$) replaces the equal sign in \eqref{eqmass}.

In \cite{Kim} the following result concerning the mass comparison is proved, which is readily seen to hold also in dimension $N=2$:
\begin{proposition}\label{preservingconc}
Assume that $\rho_1,\,\rho_2$ are respectively a subsolution and a supersolution to equation \eqref{eqmass}. Suppose that $\rho_1$, $\rho_2$ preserve the mass
through time, i.e.
\[
\int_{\R^2}\rho_1(x,t)dx=\int_{\R^2}\rho_1(x,0)dx\quad \int_{\R^2}\rho_2(x,t)dx=\int_{\R^2}\rho_2(x,0)dx
\]
and that $\rho_1$ is less concentrated than $\rho_2$ at the initial time, namely
\[
\rho_1(x,0)\prec \rho_2(x,0).
\]
Then the mass functions preserve the same order for all times:
\[
\rho_1(x,t)\prec\rho_2(x,t)\,\,\text{for all } t\geq0.
\]
\end{proposition}

It is also possible to show a two dimensional version of \cite[Theorem 5.6]{Kim}, showing an exponential convergence of the mass function of the solution to $\eqref{KellerS}$ with a
generic radial initial data to the mass function of a steady state having the same mass. Notice that the existence of a radially decreasing steady state with given
mass $M$ is guaranteed by Corollary \ref{stationary} and Theorems \ref{varthm} and \ref{regulth1}.

\begin{theorem}[Exponential convergence of the mass function]\label{uniquenessteady}
Let $\rho(x,t)$ be the solution to equation \eqref{KellerS} with initial data $\rho(x,0)\geq0$ being a continuous, radially decreasing, compactly supported
function on $\R^2$. Let $\rho_0$ be a radially decreasing steady state to \eqref{KellerS} in the distributional sense \eqref{steadystdef} with mass $M$, being $M=\|\rho(x,0)\|_{1}$.  Then
\begin{equation}
\left|M_\rho(r,t)-M_{\rho_0}(r)\right|\leq C e^{-\lambda t}\,\,\,\text{for all }r\geq0\label{asymptconv}
\end{equation}
where $C$ depends on $\rho(x,0)$, $M$, $m$, and the rate $\lambda$ only depends on $M$.
\end{theorem}
\begin{proof}
We briefly provide the main arguments of the proof, since it is totally analogous to the proof of \cite[Theorem 5.6]{Kim}, which the interested reader should refer to. We can always assume that $\rho(0,0)>0$, as otherwise $\rho(0,t)$ will become positive in finite time (see \cite[Corollary 5.5]{Kim}). It is always possible to choose a small positive constant $a$ such that
\begin{equation}
a^{2}\rho_{0}(ax)\prec\rho(x,0)\,\quad\mbox{ and }\quad a^{-2}\rho_{0}(a^{-1}x)\succ\rho(x,0).\label{eq.36}
\end{equation}
For a given nonnegative function $\xi(x,t)$ which is differentiable in its positive set, let us introduce its velocity field $\overrightarrow{v}(x,t;\xi)$ through the formula
\[
\overrightarrow{v}(x,t;\xi)=-\frac{m}{m-1}\nabla(\xi^{m-1})+\nabla(\xi\ast\mathcal{K}).
\]
It is possible to prove that if we consider the velocity field $\overrightarrow{v}(x;\rho^{a}_{1})$ of $\rho^{a}_{1}(x)=a^{2}\rho_{0}(ax)$, then its inward normal component
\[
v(r)=\overrightarrow{v}(x;\rho^{a}_{1})\cdot\left(-\frac{x}{r}\right)=\frac{m}{m-1}\frac{\partial}{\partial r}\rho^{a}_{1}-\frac{\partial}{\partial r}(\rho_{1}\ast\mathcal{K})
\]
satisfies, for all $a\in (0,1)$, the estimate
\[
v(r)\geq (1-a^{2}(m-1))a^{2}r\,\frac{M_{\rho_0}(ar)}{2\pi(ar)^{2}}\geq0.
\]
Since in the positive set of $\rho_{0}$ we have, for suitable positive constants $C_{1}$, $C_{2}$,
\begin{equation}
C_1\leq\frac{M_{\rho_0}(ar)}{2\pi(ar)^{2}}\leq C_2\label{eq.38}
\end{equation}
by the previous estimate we find
\[
v(r)\geq C_{1} (1-a^{2}(m-1))a^{2}r.
\]
With the choice of $a$ for which the two relations in \eqref{eq.36} hold, we define the function
\[
\phi(r,t)=k^{2}(t)\rho_{0}(k(t)r),
\]
where we impose that the scaling factor $k(t)$ satisfies the following ODE with initial data $k(0)=a$:
\[
k^{\prime}(t)=C_{1}(k(t))^{3}(1-(k(t))^{2(m-1)}).
\]
Then one proves that $\phi(r,t)$ is a subsolution to \eqref{eqmass} and that the following exponential estimate holds
\begin{equation}
0\leq M_{\rho_{0}}(r)-M_{\phi}(r,t)\lesssim\exp(-2C_{1}(m-1)t).\label{eq.37}
\end{equation}
Similarly, we construct a supersolution to \eqref{eqmass} by taking into account the constant $C_{2}$ in \eqref{eq.38} and defining the function
\[
\eta(r,t)=k^{2}(t)\rho_{0}(k(t)r)
\]
where $k(t)$ solves this time the following ODE with initial data $k(0)=1/a$:
\[
k^{\prime}(t)=C_{2}(k(t))^{3}(1-(k(t))^{2(m-1)}).
\]
Then $\eta(r,t)$ is shown to be a supersolution to \eqref{eqmass} whose mass function satisfies the estimate
\begin{equation}
0\leq M_{\eta}(r,t)-M_{\rho_0}(r)\lesssim\exp(-2C_{2}(m-1)t)\label{eq.39}.
\end{equation}
Now, from relations $\eqref{eq.36}$ we find $\phi(\cdot,0)\prec \rho(\cdot,0)\prec\eta(\cdot,0)$, thus by Proposition \eqref{preservingconc} we get
\begin{equation}
M_{\phi}(r,t)\leq M_{\rho}(r,t)\leq M_{\eta}(r,t) \label{eq.40}
\end{equation}
for all $r,\,t>0$. Then inequalities \eqref{eq.37}-\eqref{eq.39}-\eqref{eq.40} yield
\[
\left|M_\rho(r,t)-M_{\rho_0}(r)\right|\lesssim e^{-\lambda t}
\]
where $\lambda=2(m-1)\min\left\{C_{1},C_{2}\right\}$.
\end{proof}

As a consequence, the uniqueness of a radially decreasing steady state of a given mass $M$ follows: in fact, if there were two of such steady
states $\rho_{0}$, $\bar\rho_0$, inequality \eqref{asymptconv} ensures that
\[
M_{\rho_0}(r)=M_{\bar\rho_0}(r)\,,
\]
and therefore differentiating we find $\rho_{0}=\bar\rho_0$. Summarizing the results of the last two sections, we conclude

\begin{theorem}[Uniqueness of global minimizers]\label{uniquenessteady2}
There is a unique global minimizer of the free energy functional $\mathsf{G}$ defined by \eqref{eq.1} in $\mathcal{Y}_{M}$. Moreover, such minimizer is the
unique radially decreasing, compactly supported, and smooth in its support steady state of \eqref{KellerS}  in the distributional sense \eqref{steadystdef} characterized by \eqref{eq.33}-\eqref{eq.34}.
\end{theorem}


\section{Symmetry of the steady states}

The aim of this section is to establish the symmetry of any compactly supported steady state, not only of global minimizers, which in turn will yield the uniqueness of compactly supported steady states. Consider a nonnegative density $\rho \in \mathcal{Y}_M$ and notice that, thanks to the fact that $\rho \log \rho$ and $\rho \log (1+|x|^2)$ belong to $L^1 (\mathbb{R}^2)$, the logarithmic potential associated to $\rho$, denoted in the rest by $u = \mathcal{K}\ast\rho$, is well defined. This is for instance a consequence of the logarithmic HLS inequality \eqref{eq.2}, see \cite{CarlenLoss}. Let us specify the definition of steady state for the nonlinear diffusion Keller-Segel model \eqref{KellerS} following
\cite{Strohmer,BianLiu}.

\begin{definition}\label{stationarystates}
A nonnegative compactly supported density $\rho \in \mathcal{Y}_M$ is a \textbf{stationary state} for the evolution problem \eqref{KellerS} if $\rho\in
L^\infty(\mathbb{R}^2)$, $\rho^{m-1}\in W^{1,m}_{loc} (\mathbb{R}^2)$, and the couple $(\rho,u)$ satisfies
\begin{equation}\label{steady}
\Delta \rho^{m} = \frac{m}{m-1}\nabla\cdot\left(\rho\nabla \rho^{m-1}\right) = \nabla\cdot\left(\rho\nabla u\right) \text{ in } \mathbb{\R}^2
\end{equation}
in the distributional sense with $u$ being the Newtonian potential associated to $\rho$ as in \eqref{eq.12}.
\end{definition}

Let us first observe that the nonlinear term in the RHS of \eqref{steady} makes sense for compactly supported steady states. Notice that the logarithmic
potential $u=\mathcal{K}\ast\rho$ is a $L^1$ distributional solution of $-\Delta u = \rho$ with $\rho \in L^{m} (\mathbb{R}^2)$, $m>1$. Thus, using the elliptic
regularity theory \cite[Lemma 9.9]{Gilbarg}, we deduce that $u \in W^{2,m}_{loc} (\mathbb{R}^2)$ and, thanks to the Sobolev embedding in dimension 2, we have
that $u \in
C^{1,\alpha}_{loc} (\mathbb{R}^2)$ for some $\alpha>0$. On the other hand, by the fact that $\rho \in L^{m} (\mathbb{R}^2)$, $m>1$, and $\nabla u$ is locally
bounded, we see that the LHS of \eqref{steady} belongs to $W^{-1,m'}_{loc} (\mathbb{R}^2)$. Noticing that $\rho^{m}$ is a $L^1$ distributional solution of
\eqref{steady} with datum in $W^{-1,m'}$ and that $\rho^{m} = 0$ on the boundary of a sufficiently large ball by the compact support hypothesis, then $\rho^m$
is
in fact a weak $W_{loc}^{1,m}$ solution of \eqref{steady}, cfr. \cite{LerayLions}. Sobolev embedding shows that both $\rho^m$ and $\rho$ belong to some H\"older
space $C^{0,\alpha}(\mathbb{R}^2)$. Since $m>1$ and $\rho^{m-1}\in W^{1,m}_{loc} (\mathbb{R}^2)$, then $\nabla \rho^{m} = \frac{m}{m-1} \rho\nabla \rho^{m-1}$.
We conclude that wherever $\rho$ is positive, \eqref{steady} can be interpreted as
$$
\nabla \left( \frac{m}{m-1} \rho^{m-1} - u \right) = 0\,,
$$
in the sense of distributions in $\Omega=\mbox{supp}(\rho)$. Hence, the function $G(x)=\frac{m}{m-1}\rho^{m-1} - u(x)$ is constant in each connected component of
$\Omega$ and $u$ satisfies the elliptic equation $-\Delta u = g(x,u)$ with the nonlinearity $g$ given by
\begin{equation}\label{nonlinearita}
g(x,u) = \left(\frac{m-1}{m}\right)^{\frac{1}{m-1}} ((G(x) + u)^{+} )^{\frac{1}{m-1}}\,,
\end{equation}
for all $u\in\R$ and $x\in\Omega$. We are now ready to state our symmetry result.

\begin{theorem}\label{symmetry} Let $\rho \in \mathcal{Y}_M$ be any nonnegative compactly supported stationary state. Then $\rho$ is radially symmetric about
the
origin.
\end{theorem}

The proof of Theorem \ref{symmetry} will be achieved thanks to a non-standard Moving Plane type argument for $u$, especially thanks to a precise decay estimate
at infinity and a symmetry property for the function $G$ introduced above. This result is known in the corresponding range of nonlinearities in larger
dimensions. Here, the main technical difficulty is to deal with the logarithmic behavior of the Newtonian potential in two
dimensions.

First of all, we need to prove a precise decay estimate for $u$: we already know thanks to \eqref{asymptbehth1} that $u(x)$ behaves like $- M \log|x|$ when
$|x|$
is large, but unfortunately this is not enough for our purposes. Let us assume that $\mbox{supp}(\rho) \subset B_{r_o} (0)$, \textit{i.e.} $\rho(x)=0$ for any $|x|>r_o$:
we can refine the asymptotic behavior of $u$ as given by the following

\begin{proposition}\label{decay0}
There are $C_1, C_2>0$ such that for all $|x| \geq 2r_o$
\begin{equation}
\label{decay} |u(x) - M \mathcal{K}(x)| \leq C_1 r_o^2 |x|^{-2}
\end{equation}
and
\begin{equation}
\label{decaygrad} |\nabla (u(x) - M \mathcal{K}(x))| \leq C_2 r_o^2 |x|^{-3}
\end{equation}
hold.
\end{proposition}

\proof First of all notice that
$$
|u(x) - M \mathcal{K}(x)| = \frac{1}{2\pi} \left| \int_{\mathbb{R}^2} (\log|x-y| - \log|x|) \rho(y) \, dy \right|\,.
$$
We can proceed essentially as in the proof of \cite[Lemma 1]{Strohmer}. Notice that for $|x|\geq 2r_o$ we have $supp (\rho) \subset \{ |y| \leq \frac{|x|}{2}
\}$.
Thus, thanks to the homogeneity of the derivatives of the kernel $\mathcal{K}(x)$ and the zero center of mass condition, we have
\begin{align*}
\left|\int_{\mathbb{R}^2} (\mathcal{K}(x-y) - \mathcal{K}(x)) \rho(y) \, dy \right| &\,= \left|\int_{\mathbb{R}^2} \left( \frac{1}{2} \nabla^2
\mathcal{K}(x-\sigma y) y \cdot y - \nabla\mathcal{K}(x) \cdot y \right) \rho(y)  \, dy \right|\\
&\,\leq \frac{1}{2} \int_{|y| \leq \frac{|x|}{2}} |\nabla^2 \mathcal{K}(x-\sigma y)| |y|^2 \rho(y) \, dy \\
&\,= \frac{1}{2}  \int_{|y| \leq \frac{|x|}{2}}  \left|\nabla^2 \mathcal{K}\left(\frac{(x-\sigma y)}{|x-\sigma y|}\right)\right| \frac{|y|^2}{|x-\sigma
y|^{2}} \rho(y) \, dy\\
&\,\leq \frac{1}{2} \left( \sup_{\mathbb{S}^1} |\nabla^2 \mathcal{K}|\right) \int_{|y| \leq \frac{|x|}{2}} \frac{|y|^2}{(|x|-|y|)^{2}} \rho(y) \, dy \\
&\,
\leq  \frac{2}{|x|^2} \left( \sup_{\mathbb{S}^1} |\nabla^2 \mathcal{K}|\right) \int_{|y| \leq \frac{|x|}{2}} |y|^2 \rho(y) \, dy\\
&\,\leq 2 M r_o^2 \left( \sup_{\mathbb{S}^1} |\nabla^2 \mathcal{K}|\right) |x|^{-2}
\end{align*}
due to a simple Taylor expansion for some $0 < \sigma(y) < 1$, leading to the desired estimate \eqref{decay}. Replacing the kernel $-\log|x|$ with its $x_1$
derivative $-x_1/|x|^2$ (which is homogeneous of degree -1), the same proof leads to \eqref{decaygrad}. \qed

We will now start a moving plane type procedure in order to establish that the solution is even with respect to the first variable. Then, thanks to the
rotational invariance, we will deduce that $u$ is even with respect to any plane through the origin, which in turn means that $u$ is in fact radial. Hence, let
$x_\lambda = \sigma_\lambda (x) =(2\lambda - x_1, x_2)$ be the reflection of $x \in \Sigma_{\lambda} = \{ x_1 < \lambda \}$ with respect to the plane $T_\lambda = \{ x_1 = \lambda\}$
and let $u_\lambda (x) := u (x_\lambda) = u(2\lambda-x_1,x_2)$ be the corresponding reflection of $u$.

Thanks to the strict monotonicity of our kernel, we will start by showing that well away from the $x_2$ axis the difference of $u$ and $u_\lambda$ is
nonpositive as shown by the following:

\begin{lemma}\label{mmp1}
If $\lambda < - r_o$ then $u_\lambda (x) \geq u(x)$ for any $x \in \Sigma_\lambda$.
\end{lemma}

\proof From \eqref{decay}, for $|x| \geq 2r_o$ and $|x_{\lambda}| \geq 2r_o$ we find:
\begin{equation}\label{limit1}
u(x) - u_\lambda (x) \leq \frac{M}{2\pi} \log \frac{|x_{\lambda}|}{|x|} + C r_o^2 (|x|^{-2} + |x_{\lambda}|^{-2})
\end{equation}
For $x \in \Sigma_\lambda$ and $\lambda < 0$, we observe that $|x|-\lambda\leq |x_{\lambda}| \leq |x|$. Hence, we get
\begin{equation}\label{limit2}
\limsup_{|x|\to \infty}[ u (x) - u_\lambda (x)] \leq 0 \, \text{ for any } x \in \Sigma_\lambda\,,
\end{equation}
while $u=u_\lambda$ on $T_\lambda$ by definition. Observe that since $\rho$ is $C^{0,\alpha}$, by Schauder estimates $u$ and $u_\lambda$ satisfy respectively
$$
-\Delta u (x) = \rho(x) \quad , \quad -\Delta u_\lambda (x) = \rho(x_\lambda)
$$
in the classical pointwise sense. Moreover, since for $|x| > r$ we have that $\rho(x) = 0$ while $\rho (x_\lambda) \geq 0$ always, by \eqref{limit2} we finally
obtain
\begin{equation}\label{compar}
\Delta u \geq \Delta u_\lambda \text{ in } \Sigma_\lambda \; , \; u \leq u_\lambda \text{ on } \partial \Sigma_\lambda \; , \; \limsup_{|x|\to \infty}[ u-
u_\lambda] \leq 0.
\end{equation}
By \eqref{compar} and the classical comparison principle, we find $u \leq u_\lambda$ on $\Sigma_\lambda$.\qed

Next, we will show that the same property of Lemma \ref{mmp1}, for a fixed negative $\lambda$, is true outside a sufficiently large ball.

\begin{lemma}\label{mmp2}
For any $\lambda < 0$, there exists $R_\lambda > 0$ such that $u_\lambda (x) \geq u(x)$ for any $x \in (B (0, R_\lambda))^c \cap \Sigma_\lambda$.
\end{lemma}

\proof For any $x \in \Sigma_\lambda$, we have as above $|x|-\lambda\leq |x_{\lambda}| \leq |x|$ and thus,
$$
\lim_{|x| \to \infty} \frac{|x_\lambda|^2}{|x|^2} = 1\,.
$$
This easily implies that there exists $R_{1, \lambda}$ sufficiently large such that
\begin{equation}\label{comm}
|x_\lambda| \leq |x| \leq 2 |x_\lambda| \; \text{ for } |x| \geq R_{1, \lambda}\,.
\end{equation}
Using \eqref{limit1} in view of \eqref{comm} for $x_1 < -R_{1, \lambda}$ we have
\begin{align*}
u(x) - u_\lambda (x) &\leq \frac{M}{2\pi} \log \frac{|x_{\lambda}|}{|x|} + C r_o^2 (|x|^{-2} + |x_{\lambda}|^{-2})\\
&\leq \, \frac{M}{4\pi} \log \frac{|x_{\lambda}|^2}{|x|^2} + 5C r_o^2|x|^{-2}\\
&= \, \frac{M}{4\pi} \log \left(1+ \frac{4 \lambda (\lambda -x_1)}{|x|^2}\right) + 5C r_o^2|x|^{-2}\\
&\leq \, \frac{ M \lambda (\lambda -x_1) + 5 C \pi r_o^2}{|x|^{2}} < 0
\end{align*}
if furthermore we assume $x_1 < \lambda - \frac{5C \pi r_o^2}{M |\lambda|}$. In particular we have found that
\begin{equation}\label{segno1}
u(x) - u_\lambda (x) < 0  \text{ for } x_1 < \mu_{1,\lambda} := \min\left\{ -R_{1, \lambda}, \lambda - \frac{5C \pi r_o^2}{M |\lambda|} \right\}.
\end{equation}
Observe that by continuity and \eqref{segno1} we have that $u(x) - u_\lambda (x) \leq 0$ on $x_1 = \mu_{1,\lambda}$. Notice that if $\mu_{1, \lambda} \leq x_1 <
\lambda$ and $|x_2| \geq 2r_o$, we can apply \eqref{decaygrad} in order to get
\begin{equation}\label{decayx1}
|u_{x_1} (x) - M \mathcal{K}_{x_1} (x)| \leq C r_o^2 |x|^{-3}
\end{equation}
Recalling that $\frac{\partial}{\partial x_1} u_\lambda (x) = - u_{x_1} (2\lambda-x_1,x_2)$, choosing $|x_2| \geq \max\{ 2r_o, R_{1, \lambda}\}$, we can apply
\eqref{decayx1} and \eqref{comm} in order to deduce that
\begin{align*}
\frac{\partial}{\partial x_1} (u_\lambda (x) - u (x)) &\leq M\left(\frac{x_1}{2 \pi |x|^2}+ \frac{2 \lambda - x_1}{2 \pi |x_{\lambda}|^2}\right) + C
r_o^2(|x|^{-3} +
|x_{\lambda}|^{-3})\\
&\, \leq M \frac{\lambda}{\pi |x|^2} + M(2\lambda-x_1) \left(\frac{1}{2 \pi |x_{\lambda}|^2} - \frac{1}{2 \pi |x|^2} \right) + 9C r_o^2|x|^{-3}\\
&\, \leq M\frac{\lambda}{\pi |x|^2} +  M\frac{4\lambda(2\lambda-x_1)(x_1-\lambda)}{2 \pi |x_{\lambda}|^2 |x|^2} + 9C r_o^2|x|^{-3}\\
&\, \leq M\frac{\lambda}{\pi |x|^2} +  M\frac{8\lambda(2\lambda-\mu_{1,\lambda})(\mu_{1,\lambda}-\lambda)}{ \pi |x|^4} + 9C r_o^2|x|^{-3} =\\
&\, = M\frac{\lambda}{\pi |x|^2}\left( 1+ O \left(\frac{1}{|x|} \right)\right)<0
\end{align*}
if furthermore $|x_2| > R_{2, \lambda}$ sufficiently large. Thus, choosing $\mu_{2, \lambda} = \max \left\{ 2r_o; R_{1, \lambda};R_{2, \lambda} \right\}$, we
get
$$
(u(x) - u_\lambda (x))_{x_1}<0 \text{ for any } \mu_{1, \lambda} \leq x_1 < \lambda \text{ and } |x_2| \geq \mu_{2, \lambda}
$$
while $u \leq u_\lambda$ for $x_1=\mu_{1,\lambda}$. This in turn implies that
\begin{equation}\label{segno2}
u(x) - u_\lambda (x)<0 \text{ for any } \mu_{1, \lambda} \leq x_1 < \lambda \text{ and } |x_2| \geq \mu_{2, \lambda}.
\end{equation}
Finally, from \eqref{segno1} and \eqref{segno2}, choosing $R_\lambda := \sqrt{\mu_{1, \lambda}^{2} + \mu_{2, \lambda}^{2}}$, we obtain $u(x) - u_\lambda (x)
\leq
0$ for any $x \in (B (0, R_\lambda))^c \cap \Sigma_\lambda$. The proof of this Lemma is illustrated in Figure \ref{fig:Figura Lemma 2}. \qed

\

\begin{figure}
	\centering
		\includegraphics[width=0.70\textwidth]{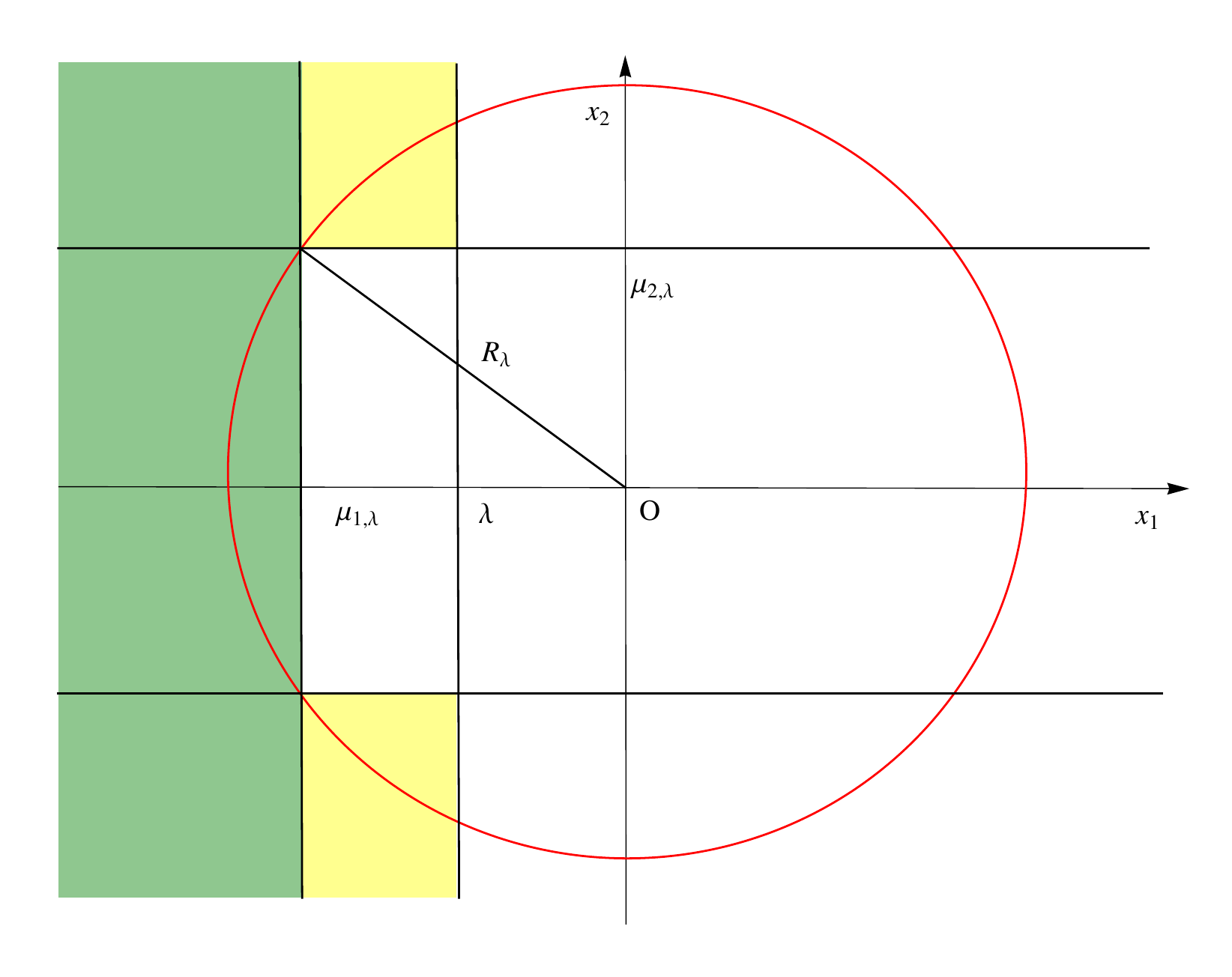}
\caption{Illustration for the estimates of Lemma \ref{mmp2}.}
\label{fig:Figura Lemma 2}
\end{figure}

Thanks to Lemma \ref{mmp1} and to the fact that $\sigma_\lambda (\overline{\Omega \cap \Sigma_\lambda})$ is empty for $\lambda < - r_0$, the following quantity is well defined,
\begin{equation}\label{lamax}
\Lambda := \sup \{\lambda < 0 : u_\lambda(x) > u(x) \text{ for any } x \in \Sigma_\lambda \text{ and } \sigma_\lambda (\overline{\Omega \cap \Sigma_\lambda}) \subset \overline{\Omega} \}\,.
\end{equation}
Moreover, by the continuity with respect to $\lambda$ and the fact that $\Sigma_\lambda$ is a decreasing set-valued function of $\lambda$, we see that $\Lambda$ is in fact attained. Our aim is then to show that $\Lambda = 0$. Hovever, since our problem is not autonomous (notice the $x$ dependence of the nonlinearity $g(x,u)$ given by \eqref{nonlinearita}), we cannot proceed with a standard moving plane argument and we need to recall from \cite{Strohmer} an important reflection property of
the function $G(x)$ introduce above.

\begin{lemma}\label{mmp3}
$G(x)=G(x_\lambda)$ for any $x \in \Sigma_\lambda$ and $\lambda \leq \Lambda$.
\end{lemma}

\proof Without loss of generality we will prove the statement for $x \in \Omega \cap \Sigma_\lambda$ and derive by continuity the property up to $\overline{\Omega \cap \Sigma_\lambda}$. By \eqref{lamax} we have that $\sigma_\lambda (\Omega \cap \Sigma_\lambda) \subset \overline{\Omega}$ and, as the reflection $\sigma_\lambda$ is a homeomorphism and sends interior points to interior points of $\Omega$, we get that $\sigma_\lambda (\Omega \cap \Sigma_\lambda) \subset \Omega$. Now, we will prove that $\sigma_\lambda (\omega \cap \Sigma_\lambda) \subset \omega$ for any component $\omega$ of $\Omega$ and for $\lambda \leq \Lambda$. Let us fix any  $x \in \omega$ and let $\mu = x_1$. Then $\sigma_\mu(x) =x$, $x \in \Sigma_\lambda$ for any $\lambda > \mu$ and $\sigma_\lambda (x) \in \Omega$ for any $\mu \leq \lambda \leq \Lambda$. Since $\sigma_\mu (x) = x \in \omega$ and the range of $\sigma_\lambda (x)$, for $\mu \leq \lambda \leq \Lambda$, is a line segment wholly inside $\Omega$, we necessarily have that $x_\lambda \in \omega$ for any $\mu \leq \lambda \leq \Lambda$, which in turn implies that $G(x) = G(x_\lambda)$. \qed

We are now ready to prove Theorem \ref{symmetry}.

\subsection{Proof of Theorem \ref{symmetry}}

We will proceed with a moving plane argument for $u(x):=(\mathcal{K}\ast\rho)(x)$. Recalling the definition of $\Lambda$ given in \eqref{lamax} as the maximal
negative parameter $\lambda$ for which the reflection $u_\lambda$ is larger than $u$, we want to prove that $\Lambda = 0$. So let us argue by contradiction and
suppose that $\Lambda < 0$. Recall that $u$ satisfies $-\Delta u = g(x,u)$ with the nonlinearity $g(x,u)$ in \eqref{nonlinearita} being nonnegative and
increasing with respect to $u$. By direct computation and thanks to Lemma \ref{mmp3},
we also get
$$
-\Delta u_{\Lambda} = g(x_{\Lambda},u_{\Lambda}) = g(x,u_{\Lambda})\,.
$$
Now, by the continuity with respect to $\lambda$ and $x$ we have that $u_\Lambda (x) \geq u(x)$ for $x \in \Sigma_\Lambda$. This implies that $g(x,u) \leq
g(x_\Lambda,u_\Lambda)$ for any $x \in \Sigma_\Lambda$, i.e. $\Delta u \geq \Delta u_\Lambda$. By the strong comparison principle, we infer that either $u \equiv
u_\Lambda$
or $u < u_\Lambda$ on $\Sigma_\Lambda$. But the case $u \equiv u_\Lambda$ can be ruled as follows. Since $u$ satisfies
\begin{equation}
-\Delta u=\rho,\label{eq.20}
\end{equation}
we have that $\rho(x)=\rho(x_{\Lambda})$, thus by the zero center of mass condition for $\rho$ and through the change of variables
$(x_1',x_2')=(2\Lambda-x_1,x_2)$, we get
\[
0=\int_{\R^{2}}x_{1}\rho(x)dx=\int_{\R^2}(2\Lambda-x^{\prime}_{1})\rho(2\Lambda-x^{\prime}_{1},x^{\prime}_{2})dx^{\prime}=
\int_{\R^2}(2\Lambda-x^{\prime}_{1})\rho(x^{\prime}_{1},x^{\prime}_{2})=2\Lambda M\,,
\]
therefore $\Lambda=0$, a contradiction. Hence $u < u_\Lambda$ on $\Sigma_\Lambda$.

For small $\varepsilon>0$ with $\Lambda+\varepsilon<0$, let us consider the cap $\Sigma_{\Lambda+\varepsilon}$ with corresponding reflection
$u_{\Lambda+\varepsilon}$.
We will proceed by dividing $\Sigma_\Lambda$ into four subsets as illustrated in Figure \ref{fig:digraph}.
Let $R_{\Lambda+\varepsilon} > 2r_o$ be as in Lemma \ref{mmp2} and consider the concentric balls $B(0,2r_o)$ and $B(0,R_{\Lambda+\varepsilon})$. We divide the
cap
$\Sigma_{\Lambda+\varepsilon}$ in four subsets given by
$$
A_1 := \overline{B(0, R_{\Lambda+\varepsilon})}^c \cap \Sigma_{\Lambda+\varepsilon}; \; A_2 := \overline{B (0, 2r_o)} \cap \Sigma_{\Lambda-\varepsilon};
$$
$$
A_3 := \overline{B (0, 2r_o)} \cap (\Sigma_{\Lambda+\varepsilon} \setminus \Sigma_{\Lambda-\varepsilon}); \; A_4 := (B (0, R_{\Lambda+\varepsilon}) \setminus
\overline{B (0, 2r_o)}) \cap \Sigma_{\Lambda+\varepsilon}
$$
\begin{figure}[ht]
\centering
\includegraphics[scale=.7]{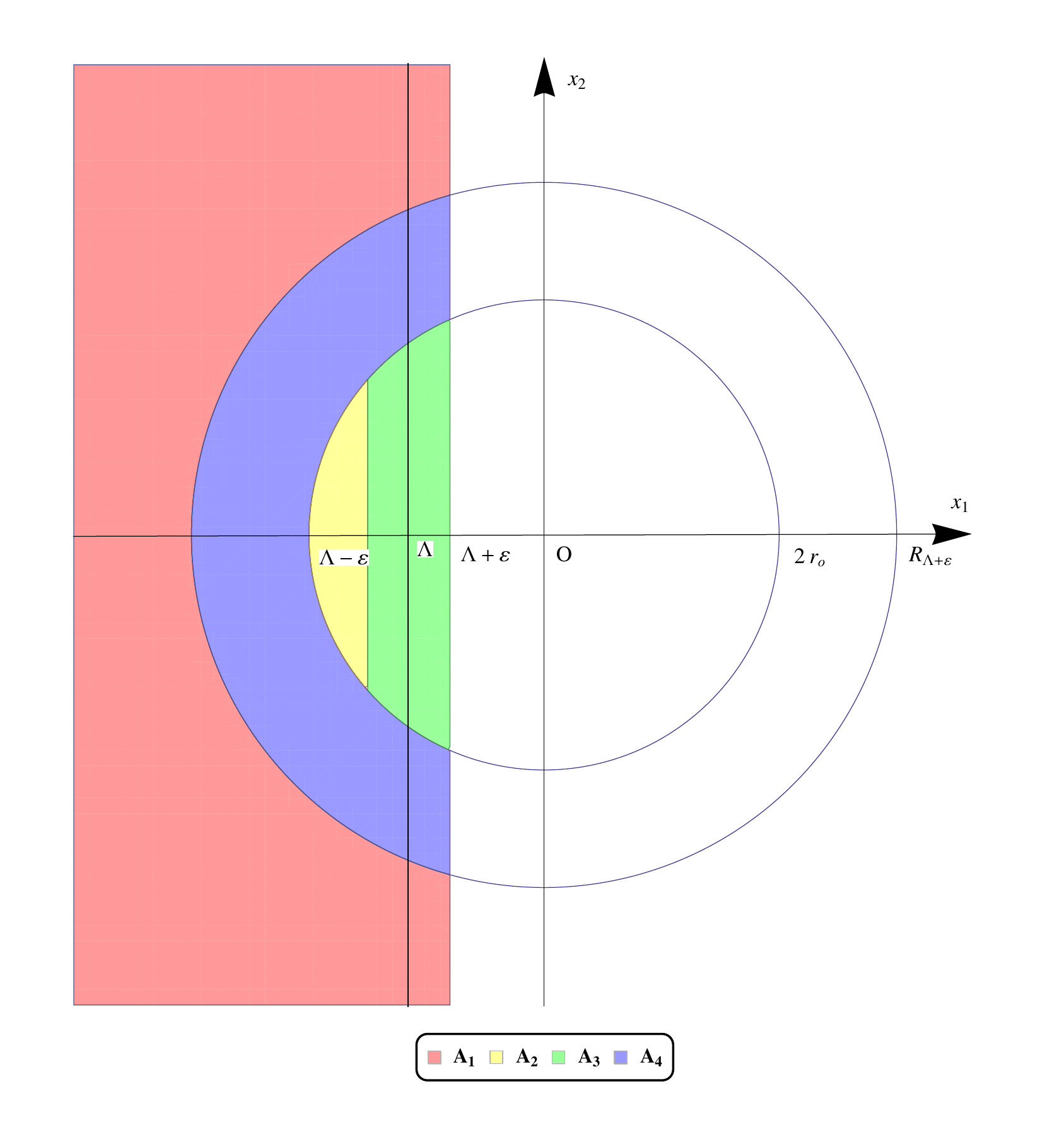}
\caption{Illustration of the division of sets for the proof of Theorem \ref{symmetry}.}
\label{fig:digraph}
\end{figure}
On the set $A_1$, we can just apply Lemma \ref{mmp2} in order to get that
$
u_{\Lambda+\varepsilon} (x) \geq u (x) \text{ for all } x \in A_1.
$
Observe that the set $A_2$ is compact, then by continuity the fact that $u_\Lambda (x) > u(x)$ implies that $u_\Lambda (x) \geq u(x) + 2\sigma$ for some small
$\sigma>0$ and for all $x \in A_2$. Thus, by the continuity with respect to $\lambda$, we see that $u_{\Lambda+\varepsilon} (x) \geq u (x) + \sigma$ for all $x
\in A_2$ for $\varepsilon$ small enough, which means
$$u_{\Lambda+\varepsilon} (x) \geq u (x) \text{ for all } x \in A_2.$$
On the set $A_3$ we need to argue as follows. Since $u < u_\Lambda$ on $\Sigma_\Lambda$ and $u = u_\Lambda$ on $T_\Lambda$, by Hopf's Lemma we know that
$\partial_\nu (u - u_\Lambda) > 0$ on $T_\Lambda$. But $\partial_\nu (u - u_\Lambda) = 2 \partial_\nu u = 2 u_{x_1}>0$ on $T_\Lambda$. In particular, there
exists a constant $\sigma_o$ such that $u_{x_1} \geq \sigma_o >0$ on $T_\Lambda \cap B(0,2r_o)$, which implies that there exists $\varepsilon>0$ small enough
such that $u_{x_1} \geq \sigma_o/2$ on $A_3$ by continuity. Then we can show that
$$
u_{\Lambda+\varepsilon}(x)-u(x)=u(2(\Lambda+\varepsilon)-x_1,x_2)-u(x_1,x_2)=\int_{x_1}^{2(\Lambda+\varepsilon)-x_1}\!\!\!\!\!\!\!\!\!\!\!\!\!\!\!\!\!\!\!\!\!
u_{x_1}(s,x_2) \,ds \geq \sigma_o (\Lambda+\varepsilon-x_1)\geq 0
$$
on $A_3$.
Finally, for any $x \in A_4$ we have $\rho(x_{\Lambda+\varepsilon}) \geq 0$ and $\rho(x)=0$ because $|x|>2r_o$. Moreover, from the sign condition proved on the
other $A_j$, $j=1,..3$, and the fact that $u_{\Lambda+\varepsilon} = u$ on $T_{\Lambda+\varepsilon}$, we have $u_{\Lambda+\varepsilon} \geq u$ on $\partial
A_4$. Thus
$$
\Delta u \geq \Delta u_{\Lambda+\varepsilon} \text{ in } A_4 \; , \; u \leq u_{\Lambda+\varepsilon} \text{ on } \partial A_4 \,.
$$
By comparison principle, we then conclude that
$u_{\Lambda+\varepsilon} (x) \geq u (x) \text{ for all } x \in A_4.$
We have thus proved that $u_{\Lambda+\varepsilon} (x) \geq u (x)$ for any $x \in \Sigma_{\Lambda+\varepsilon}$. Then, in order to contradict the maximality of $\Lambda$ as given by \eqref{lamax}, it only remains to prove that $\sigma_{\Lambda+\varepsilon} (\overline{\Omega \cap \Sigma_{\Lambda+\varepsilon}}) \subset \overline{\Omega}$. Arguing by contradiction, if this were not true there would exist two sequences $\varepsilon_k \to 0$ and $y_k \in \overline{\Omega \cap \Sigma_{\Lambda+\varepsilon_{k}}}$ such that $x_k := \sigma_{\Lambda+\varepsilon_{k}} (y_k) \notin \overline{\Omega}$. Without loss of generality, due to $\Omega$ being compact, we can suppose that $y_k \in \Omega \cap \Sigma_{\Lambda+\varepsilon_{k}}$, $y_k \to \bar{y}$ and $x_k \to \bar{x}$ as $k \to \infty$. This implies that $\bar{x} = \sigma_\Lambda(\bar{y})$ and $\bar{x} \notin \Omega$, while $\bar{y} \in \overline{\Omega \cap \Sigma_{\Lambda}}$. But by continuity and by the definition of $\Lambda$ in \eqref{lamax}, we have that necessarily $\bar{x} \in \overline{\Omega}$, and therefore $\bar{x} \in \partial \Omega$, \text{i.e.} $\rho (\bar{x})=0$. Lemma \ref{mmp3} implies that
$$
u(\bar{x}) = - G(\bar{x}) = - G(\sigma_\Lambda(\bar{y})) = - G(\bar{y}) = u(\bar{y}) - \frac{m}{m-1} \rho^{m-1}(\bar{y}) \leq u(\bar{y}) = u_\Lambda (\bar{x}) \leq u(\bar{x})
$$
From this we deduce that $\bar{x} = \sigma_\Lambda (\bar{x})$ and also $\bar{x} \in T_\Lambda = \partial \Sigma_\lambda$ which, by Hopf Lemma, implies that $u_{x_1} (\bar{x}) > 0$. Moreover, $\sigma_\Lambda (\bar{x}) = \bar{x}$, which in turn gives that $\bar{x} = \bar{y}$. Now, since $y_k \to \bar{x}$, for sufficiently large $k$ we have that $u_{x_1} (y_k) \geq \frac12 u_{x_1} (\bar{x}) > 0$. In particular, since $y_k \in \Omega$ and $\rho = \frac{m-1}{m} (C+u)^{\frac{1}{m-1}}$, there exists $\tau >0$ independent of $k$ such that $\rho(y_k + \sigma e_1) \geq \rho (y_k) > 0$ for any $k$ large and $0 \leq \sigma \leq \tau$, with $e_1 = (1,0)$ as usual. This means that $y_k + \sigma e_1 \in \Omega$ for $\sigma \in [0,\tau]$. But observe that by definition we deduce that
$$
\sigma_{\Lambda + \varepsilon_{k}} (y_k) = 2(\Lambda + \varepsilon_{k} - y_{k,1}) e_1 + y_k
$$
with $\varepsilon_{k} \searrow 0$ and $\Lambda + \varepsilon_{k} > y_{k,1} \to \Lambda$, which implies that $\sigma_{\Lambda + \varepsilon_{k}} (y_k) \in \Omega$ for $k$ sufficiently large, contradicting our assumption.

Hence $\Lambda = 0$, which implies that $u(x_1,x_2) \leq u(-x_1,x_2)$ for any $x_1 \leq 0$. Repeating the same arguments in the opposite direction, we reach
$u(x_1,x_2) \geq u(-x_1,x_2)$ for any $x_1 \geq 0$. But since
$$u(x_1,x_2) \leq u(-x_1,x_2) \leq u(-(-x_1),x_2) = u(x_1,x_2),$$
we obtain that $u$ is even in $x_1$. By rotational invariance, $u$ is even with respect to any hyperplane through the origin and hence radially symmetric. Hence,
equation \eqref{eq.20} gives the radial symmetry of $\rho$ too. $\square$

\begin{remark}[Properties of compactly supported stationary states]\label{remarkradsym}

\emph{Let $\rho$ be a stationary state to the equation \eqref{KellerS} with compact support. Then, Theorem \ref{symmetry} tells us that $\rho$ is radial around
its center of mass, assumed to be zero without loss of generality, so we have that the equation
\begin{equation}
\frac{m}{m-1}\rho^{m-1}=u+C\label{stateq}
\end{equation}
holds for some constant $C$. Arguing as in Theorem \ref{regulth1}, we find that $\rho$ satisfies the equation
\[
\frac{d}{dr}(\rho\ast\mathcal{K})(r)=\frac{du}{dr}=-\frac {M_{\rho}(r)}{2\pi r}
\]
where $M_{\rho}$ is the mass function \eqref{massfunct} of $\rho$. But then, a simple differentiation of \eqref{stateq} gives
\[
\frac{m}{m-1}\frac{d}{dr}\rho^{m-1}=-\frac {M_{\rho}(r)}{2\pi r}
\]
thus $\rho$ is radially decreasing. The same argument as in the proof of Theorem \ref{regulth1} shows that $\rho$ is smooth inside its support.}
\end{remark}

\begin{theorem} There is a unique up to translations compactly supported steady state to \eqref{KellerS} with mass $M$. Moreover, such steady state is radially
decreasing, continuous, smooth in its support, and it coincides (up to translations) with the global minimizer of the free energy functional $\mathsf{G}$ in $\mathcal{Y}_{M}$.
\end{theorem}
\begin{proof}
By Theorem \ref{symmetry} and Remark \ref{remarkradsym} we have that any compactly supported steady state $\rho$ is a radially symmetric decreasing continuous
function smooth in its support. Then Theorem \ref{uniquenessteady2} provides its uniqueness and identifies it with the global minimizer of $\mathsf{G}$ in the
class $\mathcal{Y}_{M}$.
\end{proof}

\subsection*{Acknowledgements}
JAC acknowledges support from projects MTM2011-27739-C04-02,
2009-SGR-345 from Ag\`encia de Gesti\'o d'Ajuts Universitaris i de Recerca-Generalitat de Catalunya, the Royal Society through a Wolfson Research Merit Award,
and the Engineering and Physical Sciences Research Council (UK) grant number EP/K008404/1. \smallskip\\DC acknowledges support from projects
MTM2011-27739-C04-02, 2009-SGR-345 from Ag\`en\-cia de Gesti\'o d'Ajuts Universitaris i de Recerca-Generalitat de Catalunya and by PRIN09 project \textit{Nonlinear
elliptic problems in the study of vortices and related topics} (ITALY)\smallskip.\\BV acknowledges support from the INDAM-GNAMPA project 2013 \textit{Equazioni di
evoluzione con termini non locali} (ITALY).


\end{document}